\documentclass{amsart}
\usepackage{amsmath,amsfonts,amscd,amssymb,amsthm,enumitem,xcolor}
\usepackage{mathrsfs}
\usepackage{chngpage}

\usepackage{float}
\restylefloat{table}

\def\newmathop#1{\expandafter\gdef\csname #1\endcsname{\mathop{\rm #1}\nolimits}}

\theoremstyle{theorem}

\newcounter{thmcount}

\newtheorem{theorem}[thmcount]{Theorem}
\newtheorem{corollary}[thmcount]{Corollary}
\newtheorem{lemma}[thmcount]{Lemma}

\newtheorem{conjecture}[thmcount]{Conjecture}

\newtheorem*{conjecture*}{Conjecture}

\theoremstyle{definition}

\newtheorem{remark}[thmcount]{Remark}
\newtheorem{example}[thmcount]{Example}
\newtheorem*{remark*}{Remark}
\newtheorem*{problem*}{Problem}
\newtheorem{problem}[thmcount]{Problem}
\newtheorem{definition}[thmcount]{Definition}
\newtheorem{notation}[thmcount]{Notation}

\def\Q{{\mathbb Q}}
\def\F{{\mathbb F}}

\def\Qb{\overline{\mathbb Q}}

\def\Z{{\mathbb Z}}

\def\R{{\mathbb R}}
\def\C{{\mathbb C}}

\def\cO{{\mathcal O}}

\def\Eta{{\mathrm H}}

\def\sL{{\mathscr L}}

\def\p{{\mathfrak p}}

\def\f{{\mathfrak f}}
\def\fF{{\mathfrak F}}
\def\fG{{\mathfrak G}}
\def\fg{{\mathfrak g}}

\newmathop{Disc}
\newmathop{Tr}
\newmathop{Norm}
\newmathop{ord}
\newmathop{GL}
\newmathop{SL}
\newmathop{PGL}
\newmathop{Hom}
\newmathop{Ind}
\newmathop{Res}
\newmathop{sign}
\newmathop{rk}
\newmathop{corank}
\newmathop{coker}
\newmathop{codim}
\newmathop{cyc}
\newmathop{Reg}
\newmathop{Gal}
\newmathop{Sel}
\newmathop{Frob}
\newmathop{BSD}
\newmathop{tors}
\def\triv{{\mathbf 1}}
\def\LT{\widehat{L}}

\newmathop{id}
\newmathop{Aut}
\def\fF{\mathfrak{F}}
\newcommand{\isomto}{\;\overset{\sim}{\longrightarrow}}
\newcommand{\too}{\longrightarrow}
\newcommand{\eps}{\varepsilon}
\newmathop{LT}

\input cyracc.def       
\font\tencyr=wncyr10
\def\sha{\text{\tencyr\cyracc{Sh}}}

\begin{document}

\let\introdagger\dagger

\title{On a BSD-type formula for $L$-values of Artin twists of elliptic curves}

\author{Vladimir Dokchitser, Robert Evans, Hanneke Wiersema}

\address{University College London, 25 Gordon Street, London WC1H 0AY, UK}
\email{v.dokchitser@ucl.ac.uk}
\address{King's College London, Strand, London WC2R 2LS, UK}
\email{robert.evans@kcl.ac.uk}
\address{King's College London, Strand, London WC2R 2LS, UK}
\email{hanneke.wiersema@kcl.ac.uk}

\subjclass[2010]{11G40 (11G05, 14G10)}

\begin{abstract}
This is an investigation into the possible existence and consequences of a Birch--Swinnerton-Dyer--type formula for $L$-functions of elliptic curves twisted by Artin representations. We translate expected properties of $L$-functions into purely arithmetic predictions for elliptic curves, and show that these force some peculiar properties of the Tate-Shafarevich group, which do not appear to be tractable by traditional Selmer group techniques. In particular we exhibit settings where the different $p$-primary components of the Tate--Shafarevich group do not behave independently of one another. We also give examples of ``arithmetically identical'' settings for elliptic curves twisted by Artin representations, where the associated $L$-values can nonetheless differ, in contrast to the classical Birch--Swinnerton-Dyer conjecture.
\end{abstract}

\maketitle
\tableofcontents


\section{Introduction}

The Birch--Swinnerton-Dyer conjecture classically provides a connection between the arithmetic of elliptic curves and their $L$-functions. This link is in many ways still mysterious. Indeed, some properties of $L$-functions do not obviously correspond to arithmetic properties of elliptic curves and vice versa, a classical example being the compatibility of the conjecture with isogenies, which is a highly non-trivial theorem of Cassels. In this article we focus on factorisation of $L$-functions: when $E/\Q$ is an elliptic curve and $F/\Q$ a finite extension, $L(E/F,s)$ factorises as a product of $L$-functions of twists of $E$ by Artin representations $L(E,\rho,s)$. We investigate what standard conjectures say specifically for these twisted $L$-functions.
 Ideally, we would like to give a BSD-type formula for the leading term at $s\!=\!1$ for $L(E,\rho,s)$, but, as we shall explain, there is a significant barrier to this. However, we shall provide a tool for extracting explicit arithmetic predictions, and illustrate its use by exhibiting new phenomena about the behaviour of Tate--Shafarevich groups, Selmer groups and rational points.

\subsection{BSD formula for Artin twists}

The Birch--Swinnerton-Dyer conjecture states that
$$
\ord_{s=1} L(E/F,s) = \rk E/F,
$$
and that the leading term of the Taylor series at $s\!=\!1$ of the $L$-function is given by
\[
 \lim_{s\to 1}\frac{L(E/F,s)}{(s-1)^r} \cdot\frac{\sqrt{|\Delta_F|}}{\Omega_+(E)^{r_1+r_2}|\Omega_-(E)|^{r_2}}=\frac{\Reg_{E/F}|\sha_{E/F}| C_{E/F}}{|E(F)_{\tors}|^2},
 \tag{$\dagger$}\label{bsd}
\]
where $r$ is the order of the zero, $(r_1,r_2)$ is the signature of $F$, $\Omega_{\pm}$ are the periods of $E$ and $C_{E/F}$ is the product of Tamagawa numbers and other local fudge factors from finite places (see \S\ref{ss:notation}).
Of course, the formula implicitly assumes that the Tate--Shafarevich group $\sha_{E/F}$ is finite.
We will refer to the expression on the right-hand side of \eqref{bsd} as $\BSD(E/F)$. 

Just as the Dedekind $\zeta$-function can be expressed as a product of Artin $L$-functions, so the $L$-function $L(E/F,s)$ can be written as a product of twisted $L$-functions $L(E,\rho,s)$ for Artin representations that factor through the Galois closure of $F/\Q$.
The (conjectural) analogue of the Birch--Swinnerton-Dyer rank formula is well known in this context (see e.g. \cite{Rohrlich1990} \S2):
\begin{conjecture}\label{conj:equivbsd}
For an elliptic curve $E/\Q$ and an Artin representation $\rho$ over $\Q$,
$$
\ord_{s=1} L(E,\rho,s) = \langle \rho, E(K)_\C\rangle.
$$
\end{conjecture}
\noindent Here, and throughout, $E(K)_\C=E(K)\otimes_\Z\C$ where $K$ is any finite Galois extension of $\Q$ such that $\rho$ factors through $\Gal(K/\Q)$, and $\langle\cdot, \cdot\rangle$ denotes the usual representation theoretic inner product of characters. In other words, the conjecture predicts that, for an (irreducible) $\rho,$ \linebreak the order of vanishing of $L(E,\rho,s)$ is the ``multiplicity'' of $\rho$ in the group of $K$-rational points of $E$.

However, the situation with the second part of the Birch--Swinnerton-Dyer conjecture appears to be much more difficult. 

\begin{problem}
Formulate a BSD-like formula for the leading term at $s\!=\!1$ of $L(E,\rho,s)$.
\end{problem}

There appears to be a barrier to finding such an expression, as there are ``arithmetically identical'' settings giving rise to different $L$-values. We write $\sL(E,\rho)$ for the modification of the leading term of $L(E,\rho,s)$ analogous to the left-hand side
of \eqref{bsd} (see Definition \ref{def:lalg}).

\begin{example}[see also \S\ref{s:hanneke}]
\label{ex:introha}
The elliptic curves with Cremona labels $E\!=$307a1 and $E'\!=$307c1 have the same conductor, same discriminant, no rational points, trivial Tate--Shafarevich group and trivial local Tamagawa numbers both over $\Q$ and over $\Q(\zeta_{11})^+$. However, for a Dirichlet character $\chi$ of order 5 and conductor 11, $\sL(E,\chi)\neq \sL(E',\chi)$. Specifically, $\sL(E,\chi)=1$, while $\sL(E',\chi)=(\frac{1\pm\sqrt 5}{2})^2$, the sign of $\pm\sqrt 5$ depending on the choice of~$\chi$.
\end{example}


\subsection{An arithmetic conjecture and its consequences}

We will not propose an exact expression for the hypothetical $\BSD(E,\rho)$ term for the conjectural formula
$$
\sL(E,\rho) =\BSD(E,\rho).
$$
However, based on the behaviour of $L$-functions, we will show that $\BSD(E,\rho)$ must satisfy the list of properties given in Conjecture \ref{conj:main} below. One of the roles of $\BSD(E,\rho)$ is that it lets one decompose the Birch--Swinnerton-Dyer quotient $\BSD(E/F)$ according to Artin representations, analogously to the factorisation of $L$-functions. This may at first glance look almost vacuous, but, as we will explain, the existence of such a decomposition has a range of consequences for Selmer groups, Tate--Shafarevich groups and ranks of elliptic curves.

We write $\Q(\rho)$ for the field generated by the values of the character of $\rho$, write $\rho^*$ for the dual representation, and $w_\rho$ and $w_{E,\rho}$ for the root number of $\rho$ and of the twist of $E$ by $\rho$, respectively.

\begin{conjecture}
\label{conj:main}
Let $E/\Q$ be an elliptic curve. For every Artin representation $\rho$ over $\Q$ there is an invariant $\BSD(E,\rho)\in\C^\times$ with the following properties.
Let $\rho$ and $\tau$ be Artin representations that factor through $\Gal(K/\Q).$
\begin{enumerate}[leftmargin=2em]
\item[C1.] $\BSD(E/F)=\BSD(E,\Ind_{F/\Q}\triv)$ for a number field $F$ (and $\sha_{E/F}$ is finite). 
\item[C2.]  $\BSD(E,\rho\oplus\tau)=\BSD(E,\rho)\BSD(E,\tau).$
\item[C3.] $\displaystyle\BSD(E,\rho)= \BSD(E,\rho^*)\cdot (-1)^r w_{E,\rho}w_\rho^{-2}$, where $r=\langle \rho, E(K)_\C\rangle$.
\item[C4.] If $\rho$ is self-dual, then $\BSD(E,\rho)\in \R$ and $\,\sign\BSD(E,\rho)=\sign w_\rho$.
\end{enumerate}
If $\langle \rho, E(K)_\C \rangle = 0$, then moreover:
\begin{enumerate}[leftmargin=2em]
\item[C5.] $\BSD(E,\rho)\in\Q(\rho)^\times$ and\, $\BSD(E,\rho^\fg)=\BSD(E,\rho)^\fg$ for all $\fg \in \Gal(\Q(\rho)/\Q)$.
\item[C6.] If $\rho$ is a non-trivial primitive Dirichlet character of order $d,$ and either the conductors of $E$ and $\rho$ are coprime or
$E$ is semistable and has no non-trivial isogenies over $\Q$, then $\BSD(E,\rho)\in\Z[\zeta_d].$
\end{enumerate}
\end{conjecture}

\begin{theorem}[see Corollary \ref{cor:main}]\label{lotsofconj}
Conjecture \ref{conj:main} holds assuming the analytic continuation of $L$-functions $L(E,\rho,s)$, their functional equation, the Birch--Swinnerton-Dyer conjecture, Deligne's period conjecture, Stevens's Manin constant conjecture for $E/\Q$ and the Riemann hypothesis for $L(E,\rho,s)$.
\end{theorem}

Note that the statement of the conjecture is free of $L$-functions. Morally, it should be purely a property of Selmer groups. 
However, it has some consequences that do not appear to be tractable with classical Selmer group techniques, as we now illustrate.

\begin{theorem}[see Theorem \ref{thm:makesha}, Example \ref{ex:makesha}]\label{thm:introsha}
Let $\ell$ and $p$ be primes such that the primes above $p$ in $\Q(\zeta_\ell)$ are non-principal and have residue degree 2.
If Conjecture \ref{conj:main} holds then, for every semistable elliptic curve $E/\Q$ with no non-trivial isogenies, 
$|\sha_{E/\Q}[p]|=1$ and $c_v\!=\!1$ for all rational primes~$v$, and for every cyclic extension $F/\Q$ of degree $\ell$ with $E(F)=E(\Q),$
$$
 \text{if } |\sha_{E/F}[p^\infty]|=p^2 \text{ then } |\sha_{E/F}[q^\infty]|\neq 1 \text{ for some } q\neq p.
$$ 
\end{theorem}

Roughly speaking, in the setting of the theorem the presence of the $p$-primary part of $\sha$ forces some other part of $\sha$ to be non-trivial too. It would be interesting to have a purely Selmer theoretic method that can explain such behaviour.

Conjecture \ref{conj:main} can also be used to show that purely local constraints can force certain Selmer groups of $E$ over extensions $F/\Q$ to become non-trivial. More usual methods for achieving such criteria either use Galois module structures 
or Iwasawa theoretic methods (both can be used to make $\Sel_p(E/F)$ non-trivial for $p|[F\!:\!\Q]$, see e.g. \cite{Bartel, Matsuno}) or use some form of the parity conjecture (this  requires $[F\!:\!\Q]$ to be even).

\begin{theorem}[see Corollary \ref{cor:makeselmer}]\label{thm:introsel}
Suppose Conjecture \ref{conj:main} holds.
There is an (explicit) Galois number field $F$ of odd degree and (explicit) rational prime $\ell$, such that every elliptic curve $E/\Q$ with additive reduction at $\ell$ of Kodaira type III and good reduction at other primes that ramify in $F/\Q$ has a non-trivial $p$-Selmer group $\Sel_p(E/F)$ for some prime $p\nmid [F\!:\!\Q]$.
\end{theorem}

Finally, we will also show that Conjecture \ref{conj:main} can be used to establish purely theoretical results, such as the following case of the Birch--Swinnerton-Dyer conjecture for twists of elliptic curves by dihedral Artin representations (below $D_{2pq}$ denotes the dihedral group of order~$2pq$). As far as we are aware, this does not follow from known cases of the parity conjecture.

\begin{theorem}[see Theorem \ref{thm:parityconj}]\label{thm:parityconjintro}
Let $F/\Q$ be a Galois extension with Galois group $D_{2pq}$, with $p,q\equiv 3\! \mod 4$ primes, and let $\rho$ be a faithful irreducible Artin representation that factors through $F/\Q$. If Conjecture \ref{conj:main} holds, then for every semistable elliptic curve $E/\Q$,
$$
 \text{if } \ord_{s=1} L(E,\rho,s) \text{ is odd, then } \langle\rho,E(F)_\C\rangle>0.
$$
\end{theorem}

We stress once again that Conjecture \ref{conj:main} ought to be purely a statement about Selmer groups, although we do not understand the extra structure on Selmer groups or on $\sha$ that causes it: our justification of the conjecture relies on $L$-functions. For applications like Theorem \ref{thm:parityconjintro} it is clearly important to find a proof that does not assume the Birch--Swinnerton-Dyer conjecture.

\begin{problem}
Justify Conjecture \ref{conj:main} without assuming the Birch--Swinnerton-Dyer conjecture.
\end{problem}

\begin{remark}\label{rmk:rationaltrace}
The conjecture completely determines the value of $\BSD(E,\rho)$ for Artin representations $\rho$ whose character is  
$\Q$-valued. Indeed, for a finite group $G$, the image of the Burnside ring in the rational representation ring has finite index.
Thus, if $\rho$ factors through $\Gal(K/\Q)$ where $K/\Q$ is a finite Galois extension, there are intermediate fields $F_i, F_j'$ of $K/\Q$ and a positive integer $m$ such that 
$\rho^{\oplus m}\oplus\bigoplus_i \Ind_{F_i/\Q}\triv \simeq \bigoplus_j \Ind_{F_j'/\Q}\triv$,
and so (C1), (C2) and (C4) imply that $\BSD(E,\rho)$ is the unique real number such that
$$
  \BSD(E,\rho)^m = \frac{\prod_j \BSD(E/F_j')}{\prod_i \BSD(E/F_i)} \; \; \; \text{ and } \; \; \; \sign\BSD(E,\rho)=\sign w_\rho.
$$
\end{remark}

\begin{remark}
As illustrated in Theorem \ref{thm:parityconjintro} above, our method sometimes allows us to predict the existence of points of infinite order on elliptic curves (see also \S\ref{ss:infiniteorder}, and Theorem \ref{thmquadrank} for an example with a quaternion Galois group). Needless to say, we have not found a setting where we can predict the existence of rational points which is not already predicted by the parity conjecture, 
or which outright contradicts it. The computations arising in our approach look very different from the theory of local root numbers, but (rather magically) always match.
\end{remark}


\subsection{$L$-values of Artin twists of elliptic curves} 

The heart of our approach to deriving Conjecture \ref{conj:main} lies in extracting precise consequences of $L$-function conjectures in the setting of Artin twists of elliptic curves. For the ``$L$-function side'' of the sought Birch--Swinnerton-Dyer formula for twists we use the following modification of the leading term of $L(E,\rho,s)$ at $s=1$. This is very carefully chosen so as to mesh well with the Birch--Swinnerton-Dyer conjecture over number fields, the functional equation and Deligne's period conjecture for Artin twists of elliptic curves (see \S\ref{periodconj}) at the same time. We will show that it satisfies the analogues of (C1)--(C6) of Conjecture \ref{conj:main}, which is our justification for the conjecture.

\begin{definition}\label{def:lalg}
For an elliptic curve $E/\Q$ and an Artin representation $\rho$ over $\Q$, we write
$$
  \sL(E,\rho) =  \lim_{s\to 1}\frac{L(E,\rho,s)}{(s-1)^r}  \cdot \frac{\sqrt{\f_\rho}}{\Omega_+(E)^{d^+(\rho)}|\Omega_-(E)|^{d^-(\rho)}w_\rho},
$$ 
where $r=\ord_{s=1}L(E,\rho,s)$ is the order of the zero at $s=1$, $\f_\rho$ is the conductor of $\rho$, 
and $d^{\pm}(\rho)$ are the dimensions of the $\pm 1$-eigenspaces of complex conjugation in its action on $\rho$.
\end{definition}

\begin{theorem}[Theorem \ref{thm:lfunctions}, Corollary \ref{miniprojectthm}]\label{thm:introlfunctions}
Let $E/\Q$ be an elliptic curve and let $\rho$ be an Artin representation over $\Q$.  
Fix $\zeta$ satisfying $\zeta^2=w_\rho^2w_{E,\rho}^{-1}$.

Suppose that for all Artin representations $\psi$ over $\Q$, the $L$-functions $L(E,\psi,s)$ have analytic continuation to $\C$ and satisfy the functional equation, Deligne's period conjecture and the Riemann hypothesis. Suppose also that Stevens's Manin constant conjecture holds for $E/\Q$ and the Birch--Swinnerton-Dyer conjecture holds for $E$ over number fields. Then
\begin{enumerate}[leftmargin=2.5em]
\item[(1)] $\sL(E,\Ind_{F/\Q}\triv)=\BSD(E/F)$ for a number field $F$.
\item[(2)] $\sL(E,\rho\oplus\rho')=\sL(E,\rho)\sL(E,\rho').$
\item[(3)] 
$\sL(E,\rho^*)=(-1)^r \zeta^{2} \sL(E,\rho)$,
where $r=\ord_{s=1}L(E,\rho,s).$
\item[(4)] If $\rho\simeq\rho^*$ then $\sL(E,\rho) \in \R$ and  $w_\rho\cdot\sL(E,\rho)>0$. 
\end{enumerate}

\noindent Henceforth suppose that moreover $L(E,\rho,1)\neq 0$. Then
\begin{enumerate}[leftmargin=2.5em]
\item[(5)] $\sL(E,\rho) \in \Q(\rho)$.
\item[(6)] $\sL(E,\rho)\cdot\cO_{\Q(\rho)}$ is invariant under complex conjugation as a fractional ideal of $\Q(\rho)$.
\item[(7)] $\sL(E,\rho^\fg)=\sL(E,\rho)^\fg$ for all $\fg \in \Gal(\Q(\rho)/\Q).$
\item[(8)] $\zeta$ is a root of unity. If $\f_E$ is coprime to $\f_\rho$, then $\zeta^2=(-1)^{d^{-}(\rho)}w_E^{\dim\rho}\det\rho(\f_E)$, where $\det\rho$ is regarded as a primitive Dirichlet character (see Notation \ref{not:conventions}).
\item[(9)] $\zeta\cdot \sL(E,\rho) \in \Q(\rho,\zeta)^+$; in particular $\arg\sL(E,\rho)=\arg\pm\zeta^{-1}$.
\item[(10)]\label{integrality} If $\rho$ is a non-trivial primitive Dirichlet character of order $d$, and either $\f_\rho$ is coprime to $\f_E$ or $E$ is semistable and has no non-trivial isogenies over $\Q$, then $\sL(E,\rho)\in\Z[\zeta_d].$
\end{enumerate}

\noindent Let $B=\sqrt[m]{\frac{\prod_j\BSD(E/F_j')}{\prod_i\BSD(E/F_i)}}$ for any number fields $F_i, F_j'$ and positive integer $m$ that satisfy\footnote{These exist by Remark~\ref{rmk:rationaltrace}.} 
$(\bigoplus_{\fg\in\fG}\rho^\fg)^m \oplus \bigoplus_{i}\Ind_{F_i/\Q}\triv = \bigoplus_j \Ind_{F_j'/\Q}\triv$, where $\fG=\Gal(\Q(\rho)/\Q)$. Then

\vspace{0.3em}

\begin{enumerate}[leftmargin=2.5em]
\item[(11)] $N_{\Q(\rho)/\Q}(\sL(E,\rho))=\pm B$, with sign $+$ if $m$ is odd.
\item[(12)] $N_{\Q(\rho)^+/\Q}(\zeta\cdot\sL(E,\rho)) = \pm\sqrt{B}$ if $\rho\not\simeq\rho^*$ and $\zeta \in \Q(\rho)$.
\item[(13)] $N_{\Q(\rho,\zeta)^+/\Q}(\zeta\cdot\sL(E,\rho))=\pm B$ if $\rho\not\simeq\rho^*$ and $\zeta \notin \Q(\rho)$.
\end{enumerate}
\end{theorem}

\begin{remark}
The main reason why most of the results above require the assumption that the $L$-value is non-zero is that it features in Deligne's period conjecture. 
It might be possible to extend the predictions to the higher rank case using the motivic $L$-value conjectures (Beilinson, Bloch--Kato, Equivariant Tamagawa Number Conjecture). 
These may also let one generalise the integrality statement (10) to other Artin representations and to pin down the ideal generated by $\sL(E,\rho)$ more precisely. We will not attempt to address this here.
\end{remark}



\subsection{Layout}

This paper is split into three parts. 

In \S \ref{s:robert} we extract the explicit $L$-value predictions of Theorem \ref{thm:introlfunctions} from the classical conjectures and deduce Theorem \ref{lotsofconj} from them. The key technical step here is to express the periods associated to an Artin twist of an elliptic curve to the classical periods $\Omega_{\pm}$ (Corollary~\ref{twistedperiod}).

In \S \ref{s:vladimir} we develop the arithmetic consequences for elliptic curves, including Theorems \ref{thm:introsha}, \ref{thm:introsel} and \ref{thm:parityconjintro}. The main ingredient is Theorem \ref{thm:norm}, which,  based on Conjecture \ref{conj:main}, lets us link easily controllable local invariants to ranks and the Tate--Shafarevich group. In view of Theorem~\ref{lotsofconj} these results are all consequences of the classical conjectures on $L$-functions.

In \S \ref{s:hanneke} we discuss explicit examples of $L$-values of twists of elliptic curves by Dirichlet characters and illustrate the difficulty of refining Theorem \ref{thm:introlfunctions} to a clean BSD-type prediction for the value of $\sL(E,\rho)$. We end by giving several tables of examples of a similar kind to Example \ref{ex:introha}.

We have kept the three sections largely independent of one another. In particular, the reader who does not wish to grapple with the motivic background can skip directly to the arithmetic applications in \S \ref{s:vladimir} or the $L$-value examples in \S \ref{s:hanneke}.


\subsection{Notation}\label{ss:notation}

We fix (once and for all) an algebraic closure $\Qb$ inside $\C$. All our number fields will be subfields of this choice of $\Qb$.

Formally, all our Artin representations will be $\C$-valued; that is, defined by a group homomorphism $\rho:G_K \to \Aut_\C(V)$ 
that factors through $\Gal(F/K)$ for some finite Galois extension $F/K$ and some finite dimensional complex vector space $V$. We will typically work with isomorphism classes of Artin representations, without explicitly mentioning it.
 
\medskip

\noindent The following notation is used throughout the paper:

\smallskip

\begin{tabular}{ll}
$E$ & an elliptic curve defined over $\Q$.\\
$c_v(E/F)$ & the local Tamagawa number of $E/F_v$.\\
$G_F$ & the absolute Galois group $\Gal(\Qb/F)$ of a number field $F \subseteq \Qb.$ \\
$\Frob_\p$ & (arithmetic) Frobenius element at a prime $\p$.\\
$\rho^*$ & the dual representation of an Artin representation $\rho.$ \\
$\Q(\rho)$ & the (abelian) extension of $\Q$  generated by the character values of $\rho.$ \\
$\rho^\fg$ &  for $\fg\in\Gal(\Q(\rho)/\Q)$, the Artin representation with character $\Tr\rho^{\fg}=\fg\circ \Tr\rho$.\\
$d^\pm(\rho)$ & the dimension of the $\pm1$-eigenspace of complex conjugation on $\rho.$\\
$\f_\rho$ & the Artin conductor of $\rho.$ \\
$\f_E$ & the conductor of $E/\Q$. \\
$w_\rho$ & the Artin root number of $\rho.$\\
$w_E$ & the root number of $E/\Q$ (the sign in the functional equation). \\
$w_{E,\rho}$ & the root number of the twist $E/\Q$ by $\rho$  (see \S\ref{twistsection}). \\
$\Ind_{F/K}\rho$ & $\Ind_{G_F}^{G_K}\rho$ for a field extension $F/K$ and an Artin representation $\rho$ over $F.$\\
$\ominus$ & the formal difference of Artin representaions, i.e. $\rho_1\ominus \rho_2= \rho_3 \Leftrightarrow \rho_1= \rho_2\oplus \rho_3$.\\
$\zeta_n$ & a primitive $n$-th root of unity.\\
$N_{F/K}$ & the norm map from $F$ to $K$.
\end{tabular}

\begin{notation}\label{not:conventions}
We use the convention (as in \cite{Dokchitser2005} \S3.2) that the Euler factor at a prime $p$ of $L(E,\rho,s)$ is
$$
 \det \left({\rm{Id}}-\Frob_p^{-1}p^{-s}\>\Big|\> (H^1_\ell(E) \otimes \rho)^{I_p}\right),
$$
where $I_p$ is the inertia group at $p$, $H^1_\ell(E)=H^1_{et}(E,\Q_\ell)\otimes_{\Q_\ell}\C$ for any embedding $\Q_\ell\hookrightarrow \C$ and any prime $\ell\neq p$.

To identify 1-dimensional Artin representations with Dirichlet characters, we use the isomorphism $\Gal(\Q(\zeta_n)/\Q)\to (\Z/n\Z)^\times$ given by $\sigma_a\leftrightarrow a$ for $\sigma_a:\zeta_n\to\zeta_n^a$.

We caution the reader that with these normalisations, if $L(E,s)=\sum a_n n^{-s}$ and $\chi$ is a primitive Dirichlet character of conductor coprime to that of $E$, then 
$$
 L(E,\chi,s) = \sum_{n=1}^{\infty} \overline{\chi(n)} a_n  n^{-s}.
$$
\end{notation}

\begin{notation}
For an elliptic curve $E/\Q$, we define the $\pm$-periods of $E$ to be
$$
 \Omega_+(E)=\int_{E(\C)^+}\omega \; \; \;  \text{ and } \; \; \; \Omega_-(E)=\int_{E(\C)^-}\omega,
$$
where $\omega$ is a global minimal differential on $E$ and $E(\C)^\pm$ is the set of points $P \in E(\C)$ such that $\bar{P}=\pm P$ with orientation chosen so that $\Omega_+(E) \in \R_{>0}$ and $\Omega_-(E) \in i\R_{>0}.$
\end{notation}

\begin{notation}\label{def:omega}
For an elliptic curve $E/\Q$ and a number field $F,$ we define
$$
  C_{E/F}=\prod_{v} c_v(E/F)\left|\frac{\omega}{\omega_v^\text{min}}\right|_v,
$$
where $v$ runs over the finite places of $F,$ $\omega$ is a global minimal differential for $E/\Q$ and $\omega_v^{\text{min}}$ is a minimal differential at $v$. By $\frac {\omega}{\omega_v^{\text{min}}}$ we mean any scalar $\lambda\in F^\times$ that satisfies $\omega=\lambda \omega_v^{\text{min}}$.
In terms of minimal discriminants, if $E$ is given by a Weierstrass equation $y^2+a_1xy+a_3=x^3+a_2x^2+a_4x+a_6$ with discriminant $\Delta_E$ and $\omega=\frac{dx}{2y+a_1x+a_3}$, then
$$
\Bigl|\frac{\omega}{\omega_v^\text{min}}\Bigr|_v^{-12}=\Bigl|\frac{\Delta_E}{\Delta_{E,v}^\text{min}}\Bigr|_v.
$$
\end{notation}

\begin{notation}
For an elliptic curve $E/\Q$ and a number field $F,$ we define
\[\BSD(E/F) = \frac{\Reg_{E/F}|\sha_{E/F}| C_{E/F}}{|E(F)_{\tors}|^2}.\]
\end{notation}

We also briefly recall Stevens's version of the Manin constant conjecture (\cite{Stevens} Conj. I):

\begin{conjecture}[Stevens's Manin constant conjecture]
Every elliptic curve over $\Q$ of conductor $N$ admits a modular parametrisation $X_1(N)\to E$ with Manin constant~1.
\end{conjecture} 

%

\noindent {\bf {Acknowledgements.}}
The authors would like to thank 
Chris Wuthrich for pointing out an issue in our original claim about integrality of $L$-values and for fixing it in \cite{WW}, and
David Burns for his valuable comments on a draft of the present article.
The first named author was supported by a Royal Society University Research Fellowship.
The third named author was supported by the Engineering and Physical Sciences Research Council [EP/L015234/1], through the EPSRC Centre for Doctoral Training in Geometry and Number Theory (the London School of Geometry and Number Theory) at University College London.


\section{Artin twists of elliptic curves}\label{s:robert}

\noindent In order to explain the implications of Deligne's period conjecture for Artin twists of elliptic curves, we first recall the relevant definitions from the theory of motives. We shall follow closely the presentations given in \cite[\S 4]{Coates1991} and \cite[\S 2]{Venjakob2007} and refer the reader to Deligne's article \cite{Deligne1979} for a more detailed account.

\begin{notation}
The following additional notation applies only in this section:

\vspace{0.75em}

\begin{tabular}{ll}
$\iota$ & the element of $G_\Q$ corresponding to complex conjugation. \\
$\fF$ & a number field (inside our fixed algebraic closure $\Qb,$ as always).\\
$\fF_\C$ & the ring $\fF \otimes \C$ (unadorned tensor products are over $\Q$). \\
$\fF_\ell$ & the ring $\fF\otimes \Q_\ell\simeq \prod_{\lambda|\ell}\fF_\lambda$ where $\fF_\lambda$ is the completion of $\fF$ at the prime $\lambda.$\\
$\Sigma_\fF$ & the set of real and complex embeddings $\fF \to \C.$\\
$\eps(\rho)$ & the epsilon factor of an Artin representation $\rho$ at $s=0$, i.e. $\eps(\rho)=w_\rho\sqrt{\f_\rho}.$ 
\end{tabular}
\end{notation}

\subsection{Motives}
It will be sufficient for our purposes to view motives in the naive sense; that is, as a collection of vector spaces with certain additional structures and comparison isomorphisms between them. In particular, a \textbf{(homogeneous) motive} $M$ over $\Q$ with coefficients in a number field $\fF,$ dimension $d$ and weight $w$ carries the following data:

\vspace{0.5em}

\begin{enumerate}[leftmargin=*]

\item a) A $d$-dimensional $\fF$-vector space $H_B(M)$ (the \textbf{Betti realisation}).\\
b) An $\fF$-linear involution $F_\infty$ on $H_B(M).$\\ 
c) A Hodge decomposition into free $\fF_\C$-modules:
$$
  H_B(M)\otimes \C \; = \bigoplus_{r+s=w}H^{r,s}(M)
$$
$\hphantom{c)}$ such that $F_\infty H^{r,s}(M)=H^{s,r}(M).$ \\

\item a) A $d$-dimensional $\fF$-vector space $H_{dR}(M)$ (the \textbf{de Rham realisation}).\\
b) A decreasing filtration $\{F^kH_{dR}(M) \, : \, k \in \Z\}$ of $\fF$-subspaces of $H_{dR}(M).$ \\

\item a) For each prime $\ell,$ a free $\fF_\ell$-module $H_\ell(M)$ of rank $d$ (the \textbf{$\ell$-adic realisation}).\\
b) For each prime $\ell,$ a continuous action of $G_\Q$ on $H_\ell(M).$\\

\item A comparison isomorphism between $\fF_\C$-modules
$$
  I_{M,\infty}:H_B(M)\otimes \C \isomto H_{dR}(M) \otimes \C
$$
such that $I_{M,\infty}\circ(F_\infty \otimes \iota) = (\id\otimes\, \iota)\circ I_{M,\infty}$ and
$$
  I_{M,\infty}\Big(\bigoplus_{r\geq k} H^{r,s}(M)\Big) = F^kH_{dR}(M)\otimes \C.
$$
\end{enumerate}

\begin{remark}
Comparison isomorphisms between other realisations are also part of the data carried by $M$; however, as we shall not need these for the work that follows, we choose to omit them here and refer the interested reader to sections 2.5 and 2.6 of \cite{Venjakob2007}.
\end{remark}

\subsection{Motivic $L$-functions}
Let $M$ be a motive over $\Q$ with coefficients in $\fF.$ For any prime number $\ell,$ identifying $\fF_\ell$ with $\prod_{\lambda|\ell}\fF_\lambda$ gives rise to a decomposition
$$
  H_\ell(M)=\bigoplus_{\lambda|\ell} H_\lambda(M),
$$
where $H_\lambda(M)$ is the image of $H_\ell(M)$ under scalar multiplication by unity in $\fF_\lambda.$

For each prime number $p,$ let $D_p\subseteq G_\Q$ be a choice of decomposition group at $p$ and let $I_p\subseteq D_p$ be the corresponding inertia subgroup. The \textbf{local polynomial of $M$ at $p$} is
$$
  P_p(M,t)=\det\big(\!\id-\,t\Frob_p^{-1}\mid H_\lambda(M)^{I_p}\big),
$$
where $\lambda$ is a prime of $\fF$ not lying over $p.$ We assume the standard hypothesis that $P_p(M,t)$ is independent of the choice of $\lambda$ and has coefficients in $\fF.$ For each $\sigma \in \Sigma_\fF,$ we define
$$
  L(\sigma,M,s)=\prod_p \sigma P_p(M,p^{-s})^{-1}\in \C,
$$
where $\sigma P_p(M,X)\in \sigma\fF[X]\subset \C[X]$; the expression converges for $s$ with sufficiently large real part. It is conjectured that each $L(\sigma,M,s)$ admits a meromorphic continuation to the entire complex plane which satisfies a functional equation of the form
$$
  L_\infty(M,s)L(\sigma,M,s)=\eps(\sigma,M,s)L_\infty(M^*,1-s)L(\sigma,M^*,1-s),
$$
where the Euler factor at infinity $L_\infty(M,s)$ is a product of gamma functions which does not depend on $\sigma$ (see \cite[Proposition 2.5]{Deligne1979}) and the epsilon factor $\eps(\sigma,M,s)$ is a product of a constant and an exponential (see \cite{TatN} for the details and extra hypotheses required for this construction). 

It is convenient, by identifying $\fF_\C$ with $\C^{\Sigma_\fF}$ via the canonical isomorphism of $\C$-algebras:
\[x\otimes z \longmapsto (z\sigma(x) \, : \, \sigma \in \Sigma_\fF),\]
to form a single $L$-function associated with $M$ which takes values in $\fF_\C\!:$
\[L(M,s)=(L(\sigma,M,s) \, : \, \sigma \in \Sigma_\fF).\]

\subsection{Periods}
Let $M$ be a motive over $\Q$ with coefficients in $\fF.$ For simplicity, we shall restrict to the case where $M$ has odd weight $w.$
Let $H_B(M)^\pm$ denote the $\pm1$-eigenspaces of the endomorphism $F_\infty$ and let
\[H_{dR}(M)^\pm = \frac{H_{dR}(M)}{F^{1+\lfloor w/2\rfloor}H_{dR}(M)}\]
for both choices of sign. The $\pm$\textbf{-period map} $\alpha_M^\pm$ of $M$ is the composition of the following $\fF_\C$-linear maps:
\[H_B(M)^\pm \otimes \C \too H_B(M) \otimes \C \isomto H_{dR}(M)\otimes \C \too H_{dR}(M)^\pm\otimes \C,\]
where the first map is induced by inclusion, the second map is the Betti-de Rham comparison isomorphism, and the last map is induced by the natural quotient map. It follows from \linebreak \cite[\S1.7]{Deligne1979} that $\alpha_M^\pm$ is an isomorphism. The $\pm$-\textbf{period} of $M,$ denoted by $c^\pm(M),$ is defined to be the residue class 
\[\det(\alpha_M^\pm) \text{ mod } \fF^\times\]
in $\fF_\C^\times/\fF^\times,$ where the determinant of the $\pm$-period map is calculated with respect to $\fF$-bases. As above, by identifying $\fF_\C$ with $\C^{\Sigma_\fF},$ we can also view $c^\pm(M)$ as a `tuple': 
\[c^\pm(M)=(c^\pm(\sigma,M) \in \C^\times\!/\fF^\times \, : \, \sigma \in \Sigma_\fF).\]

\subsection{Deligne's period conjecture}\label{periodconj}
Let $M$ be a motive over $\Q$ with coefficients in $\fF.$ \linebreak We retain the assumption that $M$ has odd weight. 
We say that $M$ is \textbf{critical (at $s=0$)} if, whenever $j<k$ and $H^{j,k}(M)\neq 0,$ one has $j<0$ and $k\geq 0.$
See \cite[Lemma 3]{Coates1991} for a proof that this is equivalent to the definition of criticality given in \cite{Deligne1979}.

Suppose that $M$ is critical and fix a choice of representative for the period $c^+(M)$ in $\fF_\C^\times$. Then conjectures 2.7 and 2.8 of \cite{Deligne1979} assert that
\begin{enumerate}
\item $\ord_{s=0}L(\sigma,M,s)$ is independent of $\sigma \in \Sigma_\fF$ and is non-negative.
\item If $L(M,0) \neq 0,$ there exists $x \in \fF^\times$ such that, for all $\sigma \in \Sigma_\fF,$ one has
\[L(\sigma,M,0)=\sigma(x)c^+(\sigma,M).\]
\end{enumerate}

\subsection{The motive associated to a twist}\label{twistsection}

Let $E$ be an elliptic curve over $\Q$ and let $\rho$ be an Artin representation over $\Q.$ Choose any finite abelian extension $\fF/\Q$ over which $\rho$ can be realised and let $\tau$ be an $\fF$-linear representation of $G_\Q$ such that $\C\otimes_\fF\tau \simeq \rho.$

In order to understand Deligne's period conjecture in the setting of Artin twists of elliptic curves, we are led to consider the tensor product motive $h^1(E)(1)\otimes[\tau],$ whose associated realisations and comparison isomorphisms arise by taking the tensor product of the corresponding data for the motives $h^1(E)(1)$ and $[\tau]$ (see Examples 2.1B and 2.1C in \cite{Venjakob2007} for detailed information about the latter two motives). In particular, one has that
$$
  L(h^1(E)(1)\otimes[\tau],s)=(L(E,\rho^\gamma,s+1) \, : \, \gamma \in \Gal(\fF/\Q)),
$$
where $L(E,\rho,s)$ is the Artin-twisted Hasse-Weil $L$-function whose construction is described explicitly in \cite[\S 3.2]{Dokchitser2005}. We recall (\textit{loc. cit.}) that $L(E,\rho,s)$ is conjectured to admit an analytic continuation to the whole complex plane which satisfies a functional equation of the form
$$
  \Gamma_\C(s)^{\dim\rho}L(E,\rho,s)=\eps(E,\rho,s)\Gamma_\C(2-s)^{\dim\rho}L(E,\rho^*,2-s),
$$
where $\Gamma_\C(s)=2(2\pi)^{-s}\Gamma(s)$ and the epsilon factor $\eps(E,\rho,s)$ has the form $w_{E,\rho}\cdot N_{E,\rho}^{1-s}$ where $w_{E,\rho}\in \C$ has absolute value 1 (the root number of the twist) and $N_{E,\rho}$ is a positive integer (the conductor of the twist). 
Finally, we recall that one has the following ``Artin formalism'':
\begin{enumerate}
\item $L(E,\rho_1\oplus \rho_2,s)=L(E,\rho_1,s)L(E,\rho_2,s),$ 
\item $L(E,\Ind_{F/\Q} \triv,s)=L(E/F,s),$ 
\end{enumerate}
where $L(E/F,s)$ is the usual (`un-twisted') Hasse-Weil $L$-function of $E/F.$

The following theorem will allow us to find an explicit representative for the period $c^+(h^1(E)(1)\otimes[\tau])$ in terms of the periods associated with the motives $h^1(E)(1)$ and $[\tau].$

\begin{theorem}\label{motivecalc} 
Let $M$ be a motive over $\Q$ with rational coefficients such that
\begin{enumerate}
  \item $M$ has dimension $2$ and weight $-1,$
  \item $\dim_\Q H_B(M)^\pm =1,$
  \item $H_B(M)\otimes_\Q \C = H^{0,-1}(M)\oplus H^{-1,0}(M).$
\end{enumerate}
Let $N$ be a motive over $\Q$ with coefficients in a number field $\fF$ such that
\begin{enumerate}
  \item $N$ has dimension $d$ and weight $0,$
  \item $H_B(N)\otimes_\Q \C = H^{0,0}(N).$
\end{enumerate}
Under these conditions, the motive $M\otimes N$ is critical and 
$$
  c^+(M\otimes N) = c^+(M)^{\mu} c^-(M)^{\nu} \det(I_{N,\infty}) \; \; \mathrm{ mod } \; \fF^\times,
$$
where $\mu=\dim_\fF(H_B(N)^+),$ $\nu = \dim_\fF(H_B(N)^-)$ and $\det(I_{N,\infty})$ is computed using $\fF$-bases.
\end{theorem}

\begin{proof}
The tensor product motive $M\otimes N$ is specified by the data obtained by taking the tensor product of the realisations of $M$ and $N$ and their additional structures; in particular, $M\otimes N$ is a motive of dimension $2d$ and weight $-1$ such that

\begin{enumerate}
  \item[(a)] $H_B(M\otimes N)=H_B(M)\otimes_\Q H_B(N)$ as an $\fF$-vector space.
  \item[(b)] $F_\infty(M\otimes N)=F_\infty(M)\otimes_\Q F_\infty(N)$ as an $\fF$-linear involution.
  \item[(c)] $H_{dR}(M\otimes N)=H_{dR}(M)\otimes_\Q H_{dR}(N)$ as an $\fF$-vector space.
  \item[(d)] The de Rham filtration on $H_{dR}(M\otimes N)$ is 
\[F^kH_{dR}(M\otimes N)=\left\{\begin{array}{ll}
H_{dR}(M\otimes N) & \text{ if } k\leq -1,\\
F^0H_{dR}(M) \otimes_\Q H_{dR}(N) & \text{ if } k=0,\\
0 &\text{ if } k\geq 1.
\end{array}\right.\]
  \item[(e)] The Betti-de Rham comparison isomorphism $I_{M\otimes N, \infty}$ is
  \[\big(H_B(M)\otimes_\Q H_B(N)\big)\otimes_\Q\C \xrightarrow{\;I_{M,\infty}\otimes_\C I_{N,\infty}\;}\big(H_{dR}(M)\otimes_\Q H_{dR}(N)\big)\otimes_\Q\C\]
  viewed as an isomorphism of $\fF_\C$-modules, where we have identified
  \[\big(H_B(M)\otimes_\Q H_B(N)\big)\otimes_\Q\C = \big(H_B(M)\otimes_\Q\C\big)\otimes_\C \big(H_B(N)\otimes_\Q\C\big)\]
  and similarly for the de Rham realisations.
\end{enumerate}
It follows easily from properties (a)-(d) that
\begin{align*}
H_{dR}(M\otimes N)^+&=\frac{H_{dR}(M)\otimes_\Q H_{dR}(N)}{F^0H_{dR}(M)\otimes_\Q H_{dR}(N)} \\[0.5em]
&=\frac{H_{dR}(M)}{F^0H_{dR}(M)}\otimes_\Q H_{dR}(N)\\[0.5em]
&=H_{dR}(M)^+ \otimes_\Q H_{dR}(N),\\[0.5em]
H_B(M\otimes N)^+&=\big(H_B(M)\otimes_\Q H_B(N)\big)^+\\[0.5em]
&=\big(H_B(M)^+\otimes_\Q H_B(N)^+\big)\oplus\big(H_B(M)^-\otimes_\Q H_B(N)^-\big).
\end{align*}

\noindent We choose bases for our various spaces as follows:

\begin{enumerate}
\item a $\Q$-basis $\{\gamma^+\}$ (resp. $\{\gamma^-\}$) for $H_B(M)^+$ (resp. $H_B(M)^-$),
\item a $\Q$-basis $\{\omega_0\}$ for $F^0H_{dR}(M)$ and extend to a basis $\{\omega_0,\omega_1\}$ for $H_{dR}(M),$
\item an $\fF$-basis $\{v_1^+,\ldots,v_\mu^+\}$ (resp. $\{v_1^-,\ldots,v_\nu^-\}$) for $H_B(N)^+$ (resp. $H_B(N)^-$),
\item an $\fF$-basis $\{w_1,\ldots,w_d\}$ for $H_{dR}(N).$
\end{enumerate}

\noindent In terms of these bases, we have
\begin{align*}
I_{M,\infty}(\gamma^\pm\otimes 1) &= \omega_0\otimes \eta_0^\pm + \omega_1\otimes \eta_1^\pm\, \text{ for some  } \eta_0^\pm,\eta_1^\pm \in \C, \\
I_{N,\infty}(v_j^\pm\otimes 1) &= \sum_{i=1}^d (b_{ij}^\pm\otimes\xi_{ij}^\pm)(w_j\otimes 1)\, \text{ for some  } b_{ij}^\pm\otimes\xi_{ij}^\pm \in \fF_\C,
\end{align*}
\noindent and so it follows from (e) that
\[I_{M\otimes N,\infty}(\gamma^\pm\otimes v_j^\pm\otimes 1)=\sum_{i=1}^d(b_{ij}^\pm\otimes\eta_0^\pm\xi_{ij}^\pm)(\omega_0\otimes w_i\otimes 1)+\sum_{i=1}^d (b_{ij}^\pm\otimes\eta_1^\pm\xi_{ij}^\pm)(\omega_1\otimes w_i\otimes 1).\]
Hence, with respect to these bases, the matrix of $\alpha_{M\otimes N}^+$ has $ij$th component:
\[A_{ij}=\left\{\begin{array}{ll}
(1\otimes\eta_1^+)(b_{ij}\otimes\xi^+_{ij}) & \text{if } j\leq \mu, \\
(1\otimes\eta_1^-)(b_{ij}\otimes\xi^+_{ij}) & \text{if } \mu < j \leq n,
\end{array}\right.\]
and so taking determinant yields the desired expression:
\begin{align*}
c^+(M\otimes N)&=(1\otimes \eta_1^+)^\mu(1\otimes\eta_1^-)^\nu\det(I_{N,\infty}) \; \; \text{ mod } \fF^\times\\
&=c^+(M)^\mu c^-(M)^\nu \det(I_{N,\infty}) \; \; \text{ mod } \fF^\times.
\end{align*}

\noindent Finally, to see that $M\otimes N$ is critical, we simply observe that
\[H_B(M\otimes N)\otimes \C = H^{0,-1}(M\otimes N)\oplus H^{-1,0}(M\otimes N),\]
where, viewed as $\fF_\C$-modules, we have
\begin{align*}
H^{0,-1}(M\otimes N)&=H^{0,-1}(M)\otimes_\C H^{0,0}(N),\\
H^{-1,0}(M\otimes N)&=H^{-1,0}(M)\otimes_\C H^{0,0}(N).\qedhere
\end{align*}
\end{proof}

\begin{corollary}\label{twistedperiod}
Let $E$ be an elliptic curve over $\Q$ and $\rho$ be an Artin representation over $\Q.$ Let $\fF/\Q$ be a finite abelian extension over which $\rho$ can be realised and let $\tau$ be an $\fF$-linear representation of $G_\Q$ such that $\C\otimes_\fF \tau \simeq \rho.$ Then $h^1(E)(1)\otimes[\tau]$ is a critical motive and the component of $c^+\big(h^1(E)(1)\otimes[\tau]\big)$ corresponding to our fixed embedding $\fF\subseteq \Qb\subseteq \C$ is
\[\frac{\Omega_+(E)^{d^+(\rho)}|\Omega_-(E)|^{d^-(\rho)}w_\rho}{\sqrt{\f_\rho}}\;\;\mathrm{ mod } \; \fF^\times.\]
\end{corollary}

\begin{proof}
Applying Theorem \ref{motivecalc} with $M=h^1(E)(1)$ and $N=[\tau]$ yields
\[c^+\big(h^1(E)(1)\otimes[\tau]\big)= \Omega_+(E)^{d^+(\tau)}\Omega_-(E)^{d^-(\tau)}\det(I_{\tau,\infty}) \;\; \text{ mod } \fF^\times.\]
Moreover, it follows from \cite[5.6.1]{Deligne1979} that
\[i^{d^-(\tau)}\det(I_{\tau,\infty}) = (\eps(\tau^\gamma) \, : \, \gamma \in \Gal(\fF/\Q))  \;\; \text{ mod } \fF^\times,\]
and so, since $d^\pm(\tau)=d^\pm(\rho),$ $\eps(\tau)=\eps(\rho)$ and $\Omega_-(E)^{d^-(\tau)}=i^{d^-(\tau)}|\Omega_-(E)|^{d^-(\tau)},$ we see that
\[\Omega_+(E)^{d^+(\rho)}|\Omega_-(E)|^{d^-(\rho)}\eps(\rho)\;\; \text{ mod } \fF^\times\]
is equal to the component of $c^+(h^1(E)(1)\otimes[\tau])$ corresponding to the identity in $\Gal(\fF/\Q).$
Therefore, since $\eps(\rho)=w_\rho\sqrt{\f_\rho},$ the result follows on dividing through by $\f_\rho.$ 
\end{proof}


\subsection{Properties of $\sL(E,\rho)$} 

We now turn to $L$-values and the proof of Theorem \ref{thm:introlfunctions}.

\begin{theorem}\label{thm:lfunctions}
Let $E/\Q$ be an elliptic curve and let $\rho$ and $\rho'$ be Artin representations over $\Q$. \linebreak Suppose that $L(E,\rho,s)$ and $L(E,\rho',s)$ admit an analytic continuation to $\C$.
\begin{enumerate}[leftmargin=2em]
\item[L1.] If $\rho=\Ind_{F/\Q}\triv$ for a number field $F$, then
\[\sL(E,\rho)= \lim_{s\to 1}\frac{L(E/F,s)}{(s-1)^r} \cdot \frac{\sqrt{|\Delta_F|}}{\Omega_+(E)^{r_1+r_2}|\Omega_-(E)|^{r_2}},\]
where $(r_1,r_2)$ is the signature of $F$ and $r=\ord_{s=1}L(E/F,s).$

\item[L2.] $\sL(E,\rho\oplus\rho')=\sL(E,\rho)\sL(E,\rho').$
\item[L3.] If the functional equation for $L(E,\rho,s)$ holds near $s=1,$ then
\[\sL(E,\rho)=\dfrac{(-1)^rw_{E,\rho}}{w_\rho^2}\,\sL(E,\rho^*),\]
where $r=\ord_{s=1}L(E,\rho,s).$
\item[L4.] If $\rho$ is self-dual i.e. $\rho\simeq\rho^*,$ then
\begin{enumerate}[topsep=0.35em]
\item[(i)] $\sL(E,\rho) \in \R,$
\item[(ii)] $\sign\,\sL(E,\rho)=\sign w_\rho,$ providing the Riemann Hypothesis holds for $L(E,\rho,s).$
\end{enumerate}

\item[L5.] If $L(E,\rho,1)\!\neq\!0$ and Deligne's period conjecture holds for the twist
of $E$ by $\rho$, then
\begin{enumerate}[topsep=0.5em]
\item[(i)] $\sL(E,\rho) \in \Q(\rho)^\times,$
\item[(ii)] $\sL(E,\rho^\fg)=\sL(E,\rho)^\fg$ for all $\fg \in \Gal(\Q(\rho)/\Q).$
\end{enumerate}

\item[L6.] 
If $\rho$ is a non-trivial primitive Dirichlet character of order $d$, and either
\begin{enumerate}
\item[(i)] $E$ is semistable and has no non-trivial isogenies over $\Q$, or
\item[(ii)] Stevens's Manin constant conjecture holds for $E/\Q$ and $\f_\rho$ is coprime to $\f_E$,
\end{enumerate}
then $\sL(E,\rho)\in\Z[\zeta_d].$

\end{enumerate}

\end{theorem}

\begin{proof}
For any Artin representation $\rho$ over $\Q,$ we shall denote the leading coefficient in the Taylor series expansion of $L(E,\rho,s)$ at $s=1$ by $\LT(E,\rho,1).$ In this notation, Definition \ref{def:lalg} states that
$$
 \sL(E,\rho)=\frac{\sqrt{\f_\rho}\,\LT(E,\rho,1)}{\Omega_+(E)^{d^+(\rho)}|\Omega_-(E)|^{d^-(\rho)}w_\rho}.
$$
Recall (from \cite{Martinet1977}, for example) that for an Artin representation $\rho$ over a number field $F$, the conductor $\f_{\rho}$ and root number $w_{\rho}$ are, respectively, an integral ideal of $F$ and a complex number of absolute value 1, and that they have the following formal properties: 
\begin{enumerate}[topsep=0.5em,itemsep=0.5em]
\item $\f_{\rho_1\oplus \rho_2}=\f_{\rho_1}\f_{\rho_2}\;$ and $\;w_{\rho_1\oplus\rho_2}=w_{\rho_1} w_{\rho_2},$
\item $\f_{\,\Ind_{L/F}(\rho)}=\mathrm{disc}(L/F)^{\dim\rho} N_{L/F}(\f_\rho)\;$ and $\;w_{\,\Ind_{L/F}(\rho)}=w_\rho.$
\end{enumerate}
We refer to these as ``Artin formalism'', analogously to the case of $L$-functions given in~\S\ref{twistsection}. 

\vspace{0.5em}

L1. By Artin formalism for the other factors, it suffices to prove that
\[d^+(\Ind_{F/\Q}\triv)=r_1+r_2 \;\;\; \text{ and } \;\;\; d^-(\Ind_{F/\Q}\triv)=r_2.\]
Since $\Ind_{F/\Q}\triv$ is the permutation module $G_\Q/G_F,$
we have
\[d^++d^-=[F:\Q] \;\;\; \text{ and } \;\;\; d^+-d^-=\,\text{tr}\big(\!\Ind_{F/\Q}(\triv)(\iota)\big),\]
where $\iota \in G_\Q$ is complex conjugation. However, we also have that
\[\text{tr}\big(\!\Ind_{F/\Q}(\triv)(\iota)\big)= \#\text{ singleton orbits of } \iota \text{ in } G_\Q/G_F=r_1,\]
and so $d^++d^-=r_1+2r_2$ and $d^+-d^-=r_1$ and the claim now follows.

\vspace{0.5em}

\noindent L2. By Artin formalism for the other factors, it suffices to note the identity:
\[d^\pm(\rho\oplus\rho')=d^\pm(\rho)+d^\pm(\rho').\]


\noindent L3. Applying $\left.\dfrac{d^r}{ds^r}\right|_{s=1}$ to the functional equation for $L(E,\rho,s)$ yields
\[\LT(E,\rho,1)=w_{E,\rho}\cdot(-1)^r\LT(E,\rho^*,1),\]
and so, since $\f_{\rho^*}=\f_\rho$ and $d^\pm(\rho^*)=d^\pm(\rho),$ we have

\[\frac{\sqrt{\f_{\rho}}\LT(E,\rho,1)}{\Omega_+(E)^{d^+(\rho)}|\Omega_-(E)|^{d^-(\rho)}w_\rho}=(-1)^r\frac{w_{E,\rho}w_{\rho^*}}{w_\rho}\frac{\sqrt{\f_{\rho^*}}\LT(E,\rho^*,1)}{\Omega_+(E)^{d^+(\rho^*)}|\Omega_-(E)|^{d^-(\rho^*)}w_{\rho^*}}\]

\noindent which, on recalling that
$w_\rho w_{\rho^*}=w_{\rho\oplus \rho^*}=1,$ simplifies to the given formula. 


\vspace{0.5em}

\noindent L4. For $s\in \C$ with Re$(s)\gg 1,$ $L(E,\rho,s)$ can be expressed in terms of its Dirichlet series and since the character of $\rho$ is real-valued it follows that the coefficients of this series are real. In particular, one has $L(E,\rho,s)=\overline{L(E,\rho,\overline{s})}$ for all sufficiently large $s \in \R.$ Since $L(E,\rho,s)$ is analytic everywhere in $\C$ (conjecturally), it follows that $L(E,\rho,s)=\overline{L(E,\rho,\overline{s})}$ for all $s \in \C$ and so, in particular, that $\LT(E,\rho,1) \in \R.$ Moreover, since $\rho$ is self-dual, we have $w_\rho=\pm1$ and so it follows that $\sL(E,\rho) \in \R.$

Assuming the Riemann Hypothesis holds, we have that $L(E,\rho,s)\neq 0$ for all $s \in (1,\infty).$ Moreover, by looking at the Euler product, it is clear that $L(E,\rho,s)\geq 0$ for all sufficiently large $s \in \R$ and so, since $L(E,\rho,s)$ is continuous on $[1,\infty),$ the intermediate value theorem implies that $\LT(E,\rho,1)> 0.$ It thus follows that one has $\sL(E,\rho)w_\rho \in \R_{>0}.$

\vspace{0.5em}

\noindent L5. As in \S \ref{twistsection}, let $\fF/\Q$ be a finite abelian extension over which $\rho$ can be realised and let $\tau$ be an $\fF$-linear representation of $G_\Q$ such that $\C\otimes_\fF \tau \simeq \rho.$ By Corollary \ref{twistedperiod}, the motive $h^1(E)(1)\otimes[\tau]$ is critical and moreover, if $L(E,\rho,1)\neq0$ and Deligne's period conjecture holds, then

\begin{enumerate}

\item[(i)] $\sL(E,\rho) \in \fF^\times,$

\item [(ii)] $\sL(E,\rho^\gamma)=\sL(E,\rho)^\gamma$ for all $\gamma \in \Gal(\fF/\Q).$

\end{enumerate}

\noindent The result follows from this on noting that $\rho^\gamma \simeq \rho$ for all $\gamma \in \Gal(\fF/\Q(\rho)).$

\vspace{0.5em}


\noindent L6. 
(i) This follows directly from \cite[Theorem 1]{WW}. 
(ii) In this case $f_{\rho}$ is not divisible by any prime where $E$ has bad reduction, so the claim follows from \cite[Theorem 2a]{WW}.
\end{proof}

\begin{corollary}\label{cor:main}
Let $E$ be an elliptic curve over $\Q.$ Suppose that $L(E,\rho,s)$ has an analytic continuation to $\C$ for all Artin representations $\rho$ over $\Q$, and set $\BSD(E,\rho)=\sL(E,\rho)$. Then C1-C6 of Conjecture \ref{conj:main} hold subject to the following conditions:
\begin{enumerate}[leftmargin=2em]
\item[C1.] The Birch--Swinnerton-Dyer conjecture holds for $E$ over number fields.
\item[C2.] Unconditional.
\item[C3.] $L(E,\rho,s)$ satisfies the functional equation and Conjecture \ref{conj:equivbsd}.
\item[C4.] $L(E,\rho,s)$ satisfies the Riemann hypothesis.
\item[C5.] $L(E,\rho,s)$ satisfies Deligne's period conjecture.
\item[C6.] Stevens's Manin constant conjecture holds for $E/\Q$.
\end{enumerate}
\end{corollary}

In the following corollary we prove the remaining parts of Theorem \ref{thm:introlfunctions}. In particular, parts (1)-(4) record some of the formal consequences of the (conjectural) properties L1-L5 stated in Theorem \ref{thm:lfunctions} and parts (5)-(7) use the classical Birch$-$Swinnerton-Dyer conjecture to make predictions about the norm of $\sL(E,\rho).$

\begin{corollary}\label{miniprojectthm}
Let $E/\Q$ be an elliptic curve and let $\rho$ be an Artin representation over $\Q.$
Suppose that $L(E,\rho,s)$ admits an analytic continuation to all of $\C$ and that $L(E,\rho,1)\neq 0.$ Suppose moreover that $L(E,\rho,s)$ satisfies the functional equation, the Riemann Hypothesis and Deligne's period conjecture. Then

\begin{enumerate}[leftmargin=*]

\item $U(E,\rho):=\dfrac{\sL(E,\rho^*)}{\sL(E,\rho)}$ is a root of unity in $\Q(\rho).$

\item $\sL(E,\rho)\cdot\cO_{\Q(\rho)}$ is invariant as a fractional ideal under complex conjugation.

\item $\zeta\cdot \sL(E,\rho) \in \Q(\rho,\zeta)^+$ where $\zeta \in \C$ is such that $\zeta^2=U(E,\rho).$

\item $U(E,\rho)=w_\rho^2 w_{E,\rho}^{-1}$ and, if $\f_E$ is coprime to $\f_\rho,$ this equals $(-1)^{d^-(\rho)}w_E^{\dim\rho}\det\rho(\f_E),$
where $\det\rho$ is regarded as a primitive Dirichlet character.
\end{enumerate}

\noindent Let $\fG=\Gal(\Q(\rho)/\Q)$ and choose number fields $F_i,F_j'$ and a positive integer $m$ such that
$$
\Big(\bigoplus_{\fg\in\fG}\rho^\fg\Big)^{\oplus m}\oplus\,\bigoplus_i \Ind_{F_i/\Q}\triv \;\simeq\; \bigoplus_j \Ind_{F_j'/\Q}\triv
$$
(these exist by Remark \ref{rmk:rationaltrace}),
and write $B$ for the unique positive real number such that
\[B^m=\frac{\prod_j\BSD(E/F_j')}{\prod_i\BSD(E/F_i)}.\]
Suppose that the Birch$-$Swinnerton-Dyer conjecture holds for $E$ over number fields. Then

\begin{enumerate}[leftmargin=*]

\item[(5)] $N_{\Q(\rho)/\Q}(\sL(E,\rho))=\pm B$, with sign $+$ if $m$ is odd.

\item[(6)] If $\rho\not\simeq\rho^*$ and $\zeta \in \Q(\rho),$ then $N_{\Q(\rho)^+/\Q}(\zeta\cdot\sL(E,\rho)) = \pm\sqrt{B}.$

\item[(7)] If $\rho\not\simeq\rho^*$ and $\zeta \notin \Q(\rho),$ then $N_{\Q(\rho,\zeta)^+/\Q}(\zeta\cdot\sL(E,\rho))=\pm B.$

\end{enumerate}
\end{corollary}

\begin{proof}

1) It follows from L5 that $U(E,\rho) \in\Q(\rho)$ and that
$$
  U(E,\rho)^\fg = \overline{\sL(E,\rho^\fg)}/\sL(E,\rho^\fg) \; \; \; \forall \fg \in \fG.
$$
In particular, one has $|U(E,\rho)^\fg|=1$ for all $\fg \in\fG,$ and so $U(E,\rho)$ is indeed a root of unity.

\smallskip

\noindent 2) This follows directly from (1) on noting that, by L5, one has $\sL(E,\rho^*)=\overline{\sL(E,\rho)}$.

\smallskip

\noindent 3) It follows from L2 and L5 that
\[\sL(E,\rho^\fg\oplus (\rho^\fg)^*)=U(E,\rho^\fg)\sL(E,\rho^\fg)^2=\left(U(E,\rho)\sL(E,\rho)^2\right)^\fg \; \; \; \forall\fg \in \fG,\]
and so, recalling that $w_{\rho^\fg\oplus(\rho^\fg)^*}=1$ for any $\fg \in \fG,$ it follows from L4 that
\[(U(E,\rho)\sL(E,\rho)^2)^\fg \in \R_{>0} \; \; \; \forall \fg \in \fG.\]
Hence, taking $\zeta \in \C$ such that $\zeta^2=U(E,\rho),$ we see that $\zeta\cdot \sL(E,\rho) \in \Q(\rho,\zeta)^+.$

\smallskip

\noindent 4) The first statement follows immediately from L3 and the second from \cite[Theorem 16]{Dokchitser2005}.

\smallskip

\noindent 5) Writing $R(\rho)=\bigoplus_{\fg \in \fG}\rho^\fg,$ it follows from L2 and L5 that 
\[\sL(E,R(\rho))=N_{\Q(\rho)/\Q}(\sL(E,\rho)).\]
However, since $R(\rho)^{\oplus m}\oplus \bigoplus_i \Ind_{F_i/\Q}\triv \simeq \bigoplus_j \Ind_{F_j'/\Q}\triv,$ another application of L2 gives
\[\sL(E,R(\rho))^m=\frac{\prod_j\sL(E,\Ind_{F_j'/\Q}\triv)}{\prod_i\sL(E,\Ind_{F_i/\Q}\triv)}.\]
Hence, by L1 together with the Birch$-$Swinnerton-Dyer conjecture, we get
\[\sL(E,R(\rho))^m=B^m,\]
and so the result follows on taking $m$th roots.

\vspace{0.75em}

\noindent 6)\,\&\,7) Let $\Gamma=\Gal(\Q(\rho,\zeta)/\Q)$ and $\Eta=\Gal(\Q(\rho,\zeta)^+/\Q).$ If $\rho\not\simeq\rho^*,$ then
$$
  \bigoplus_{\gamma \in \Gamma} \rho^\gamma \; \simeq \; \bigoplus_{\eta \in \Eta}\,(\rho\oplus\rho^*)^\eta,
$$
and so it follows from L2, L3, L4 and L5 that
\[\sL\Big(E,\;\bigoplus_{\gamma \in \Gamma}\rho^\gamma\Big) = N_{\Q(\rho,\zeta)^+/\Q}(\sL(E,\rho\oplus \rho^*)) = N_{\Q(\rho,\zeta)^+/\Q}(\zeta\cdot\sL(E,\rho))^2.\]
On the other hand, L2 gives us
\[\sL\Big(E,\;\bigoplus_{\gamma \in \Gamma}\rho^\gamma\Big) = \left\{\begin{array}{ll}
\sL(E,R(\rho)) & \text{if } \zeta \in \Q(\rho), \\
\sL(E,R(\rho))^2 & \text{if } \zeta \notin \Q(\rho),
\end{array}\right.\]
and so the results follow from L1 together with the Birch$-$Swinnerton-Dyer conjecture. 
\end{proof}


\section{Arithmetic applications}\label{s:vladimir}

In order to obtain arithmetic applications of Conjecture \ref{conj:main} we shall make use of (C5), which is the analogue of the Galois equivariance property of $L$-values. As we do not have an exact expression for $\BSD(E,\rho)$ in general, we shall take the following approach. 
The representation $\tau=\bigoplus_{\fg\in\Gal(\Q(\rho)/\Q)}\rho^\fg$ has rational trace and hence $\BSD(E,\tau)$ can, 
on the one hand, be expressed in terms of BSD-quotients $\BSD(E/K_i)$ for suitable fields $K_i$ (see Remark \ref{rmk:rationaltrace}), and, on the other hand, is the norm of $\BSD(E,\rho)$ from $\Q(\rho)$ to $\Q$ by (C5). As we shall illustrate, this places non-trivial constraints on the $\BSD(E/K_i)$, and hence on ranks and the Tate--Shafarevich groups. We stress that the expression is the norm of an {\em element} of $\Q(\rho)$, rather than just of a fractional ideal.

\begin{theorem}\label{thm:norm}
Let $G$ be a finite group and $\rho$ an irreducible representation. Write
$$
  \qquad \Big(\bigoplus_{\fg\in\Gal(\Q(\rho)/\Q)} \!\!\!\rho^{\fg}\Big)^{\oplus m} 
  = \bigl(\bigoplus_i \Ind_{H_i}^G \triv\bigr) \ominus \bigl(\bigoplus_j \Ind_{H_j'}^G \triv\bigr),
$$
for some $m\in\Z$ and some subgroups $H_i, H_j'<G$.
If Conjecture \ref{conj:main} (C1, C2, C5) hold, then for every elliptic curve $E/\Q$ and Galois extension $F/\Q$ with Galois group $G$, either
$$
 \langle\rho,E(F)_\C\rangle > 0
$$
or
$$
  \frac{\prod_i \BSD(E/F^{H_i})}{\prod_j \BSD(E/F^{H_j'})} = N_{\Q(\rho)/\Q}(x)^m
$$
for some $x\in\Q(\rho)$. 
Moreover, if $\rho$ is non-trivial, $G$ is abelian of exponent $d$ and (C6) of Conjecture \ref{conj:main} holds, then one can take $x\in\Z[\zeta_d]$ provided that either $\f_E$ and $\f_\rho$ are coprime, or $E$ is semistable and has no non-trivial isogenies.
\end{theorem}

\begin{proof}
If $\langle\rho, E(F)_\C\rangle=0$, then the formula is satisfied by $x = \BSD(E,\rho)$.
\end{proof}

In order to make use of the above theorem in specific settings, we will need to control the various terms in the $\BSD$ factors of the formula. For convenience of the reader we have recorded in \S\ref{ss:basic} some standard facts about Selmer groups, Tamagawa numbers and the term $|\omega/\omega^{\min}|$ that will be used in our computations.


\subsection{Interplay between $p$-primary parts of the Tate-Shafarevich group}

For our first application we will take the simplest setting, when the Galois group is cyclic of prime order, and make use of the fact that the ratio of $\BSD$-terms is the norm of a {\em principal} ideal. 
The basic  idea is that if the $p$-part of this number cannot be expressed as the norm of a principal ideal, then, necessarily, the $q$-primary part must be non-trivial for some other prime $q$.

\begin{theorem}\label{thm:makesha}
Let $\ell$ and $p$ be primes such that the primes above $p$ in $\Q(\zeta_\ell)$ are non-principal and have residue degree 2.
If Conjecture \ref{conj:main} holds then for every semistable elliptic curve $E/\Q$ with no non-trivial isogenies, with $|\sha_{E/\Q}[p]|\!=\!1$ and $c_v\!=\!1$ for all rational primes $v$, and for every cyclic extension $F/\Q$ of degree $\ell$ with $E(F)=E(\Q)$,
$$
 \text{if } |\sha_{E/F}[p^\infty]|=p^2 \text{ then } |\sha_{E/F}[q^\infty]|\neq 1 \text{ for some } q\neq p, \ell.
$$ 
\end{theorem}

\begin{proof}
We will in fact prove the stronger statement that $\sha_{E/\Q}[q^\infty]$ is strictly smaller than $\sha_{E/F}[q^\infty]$ for some $q\neq p, \ell$.  The fact that it is a subgroup for all $q\neq \ell$ is standard: 
it is true for $q^k$-Selmer groups by Lemma \ref{lem:basics}, and as $E(\Q)=E(F)$ it is also true for $\sha[q^\infty]$.

A 1-dimensional faithful representation $\chi$ of $\Gal(F/\Q)\simeq C_\ell$ has $\Q(\chi)=\Q(\zeta_\ell)$. Now
$$
 \sum_{\fg\in\Gal(\Q(\chi)/\Q)} \chi^{\fg} = \C[G] \ominus \triv,
$$
so by Theorem \ref{thm:norm}
$$
 \frac{\BSD(E/F)}{\BSD(E/\Q)}= N_{\Q(\zeta_\ell)/\Q}(x)
$$
for some $x\in \Z[\zeta_\ell]$. As $E(F)=E(\Q)$, it follows that $E(F)_{\tors}=E(\Q)_{\tors}$ and $\Reg_{E/F}=\ell^{\rk E/\Q} \Reg_{E/\Q}$ (Lemma \ref{lem:basics}(\ref{lem:basics-height})). 
As $E/\Q$ is semistable, the contributions to $C_{E/F}$ and $C_{E/\Q}$ only come from Tamagawa numbers  (Lemma \ref{lem:basics}(\ref{lem:basics-semistableomega})). These are trivial over $\Q$ by hypothesis, so are also trivial at all primes that split in $F/\Q$; if $v$ is a prime of multiplicative reduction that does not split then the corresponding Tamagawa number over $F$ is 1 unless the reduction is split multiplicative and the prime ramifies, in which case it is $\ell$ (Lemma \ref{lem:basics}(\ref{lem:basics-semistabletamagawa})). Putting this together we deduce that 
$$
  \qquad\frac{|\sha_{E/F}|}{|\sha_{E/\Q}|} \cdot \ell^{n} = N_{\Q(\zeta_\ell)/\Q}(x) \qquad \text{for some } n\in\Z.
$$
Note that $\ell$ is totally ramified in $\Q(\zeta_\ell)$ and the ideal above it $(\zeta_\ell-1)$ is principal. Thus if $\sha_{E/F}[q^\infty]=\sha_{E/\Q}[q^\infty]$ for all $q\neq p,\ell$, it would follow that $p^2$ is the norm of a principal ideal of $\Z[\zeta_\ell]$. By assumption, the primes above $p$ have norm $p^2$ and are non-principal, so this is not the case.
\end{proof}

\begin{example}\label{ex:makesha}
Let $E/\Q$ be a semistable elliptic curve with no non-trivial isogenies,
with $\prod_v c_v=1$ and $\sha_{E/\Q}\!=\!1$, and let $F/\Q$ be the degree 229 subfield of $\Q(\zeta_{2749})$. In this setting, 
$$
 \text{if} \qquad E(\Q)=E(F) \qquad \text{then} \qquad |\sha (E/F)|\neq p^2, 
$$
for the prime $p=1148663$.
Indeed, $p$ is a prime which has residue degree 2 in $\Q(\zeta_{229})$, and the prime above it is 
non-principal (this is hard to achieve, which is why $\ell=229$ is so large), so this is a consequence of the above theorem. 
However, it is perfectly possible for such a curve to have
$$
  E(\Q)=E(F) \qquad \text{and} \qquad |\sha (E/F)|= p^2\times (\text{integer coprime to } p),
$$
as, for instance, is the case for the elliptic curve 2749a1. (This is based on a Magma computation of the analytic order of $\sha$ and the analytic rank, and assumes the BSD conjecture.) 
\end{example}


\subsection{Forcing non-trivial Selmer groups}

The $\BSD$-terms $\BSD(E/F^{H_i})$ in Theorem \ref{thm:norm} are composed of ``hard'' global invariants (Tate--Shafarevich group and points of infinite order) and ``easy'' local invariants (Tamagawa numbers and differentals). We will now illustrate how the result can be used to make the easy local data force non-trivial behaviour of global invariants. Once again, we will exploit the fact that the ratio of $\BSD$-terms is the norm of a principal ideal. We focus on non-abelian groups of the form $G=C_{q_1q_2}\!\rtimes\! C_r$, and begin by simplifying the norm condition.

\begin{theorem}\label{thm:qqm}
Let $q_1, q_2, r$ be distinct odd primes with $q_1, q_2 \equiv 1\mod r$, but $q_1, q_2 \not\equiv 1\mod r^2$, and with $q_1$ an $r^{\text th}$ power in $\Z/q_2\Z$ and vice versa.
Let $G=C_{q_1q_2}\!\rtimes\! C_r$ with 
$C_r$ acting non-trivially on both the $C_{q_1}$ and $C_{q_2}$ subgroups.

If Conjecture \ref{conj:main} holds, 
then for every elliptic curve $E/\Q$ and every Galois extension $F/\Q$ with Galois group $G$, either $\rk E/F>0$ or
$$
\frac{\BSD(E/F^{C_r})\BSD(E/\Q)}{\BSD(E/F^{C_{q_1}\rtimes C_r})\BSD(E/F^{C_{q_2}\rtimes C_r})} = q_1^a q_2^b m
$$
for some $a, b\in \Z$ and $m\in\Q$ with $m\equiv 1 \!\mod q_1q_2$.
\end{theorem}

\begin{proof}
Let $\psi$ be a faithful 1-dimensional representation of $C_{q_1q_2}$ and $\rho=\Ind_{C_{q_1q_2}}^G\psi$. This is a faithful $r$-dimensional irreducible representation of $G$, and $\Q(\rho)=\Q(\zeta_{q_1q_2})^{C_r}$ with the $C_r$ action coming from $C_r\subset \Aut(C_{q_1q_2})=(\Z/q_1q_2\Z)^\times=\Gal(\Q(\zeta_{q_1q_2})/\Q)$. 
Applying Theorem~\ref{thm:norm} to the identity
$$
 \bigoplus_{\fg\in\Gal(\Q(\rho)/\Q)} \rho^{\fg}\qquad = \qquad \Ind_{C_r}^G\triv\> \> \oplus\> \triv \ominus\> \Ind^G_{C_{q_1}\!\rtimes C_r}\triv \>\ominus\> \Ind^G_{C_{q_2}\!\rtimes C_r}\triv
$$
shows that either $\rk E/F>0$ or the ratio of the $\BSD$ terms in the statement is a norm of an element $x\in \Q(\rho)$.

It thus suffices to show that the norm of every non-zero principal ideal $(x)$ in $\Q(\rho)$ is of the form $q_1^a q_2^b m$, for some $a, b\in \Z$ and $m\in\Q$ with $m\equiv 1 \!\mod q_1q_2$. Observe that $\Q(\zeta_{q_1q_2})=\Q(\rho)(\zeta_{q_1})=\Q(\rho)(\zeta_{q_2})$. In particular, $\Q(\zeta_{q_1q_2})/\Q(\rho)$ is unramified at all primes, so by class field theory
$$ 
\prod_\p \Frob_\p^{\ord_\p x} = \text{id} \in \Gal(\Q(\zeta_{q_1q_2})/\Q(\rho)),
$$
the product taken over the primes of $\Q(\rho)$ (all the infinite places being complex as $r$ is odd). By hypothesis $q_1$ is an $r^{\text{th}}$ power in $\F_{q_2}$ and $r^2\nmid |\F_{q_2}^\times|$, so $\F_{q_2}(\zeta_{q_1})/\F_{q_2}$ has degree coprime to $r$; hence every prime above $q_2$ must split in $\Q(\rho)(\zeta_{q_1})/\Q(\rho)$, that is $\Frob_\p=$id for every $\p|q_2$. Similarly $\Frob_\p=$id for every $\p|q_1$, and hence
$
\prod_{\p| q_1, q_2} \Frob_\p^{\ord_\p x} = \text{id}.
$
For a prime $\p\nmid q_1,q_2$ the Frobenius element is simply $\zeta_{q_1q_2}\mapsto \zeta_{q_1q_2}^{N(\p)}$, which shows that
$$
  \prod_{\p\nmid q_1, q_2} N(\p)^{\ord_\p x} = 1 \in (\Z/q_1q_2\Z)^\times,
$$
and hence the norm of $x$ is of the required form.
\end{proof}

\begin{corollary} \label{cor:makeselmer}
Let $F/\Q$ be a Galois extension of degree $3q_1q_2$ that contains a Galois cubic field $K/\Q$ and $\ell$ a prime that satisfy
\begin{itemize}
\item $q_1,q_2 \equiv 4$ or $7 \bmod 9$, and $q_1$ is a cube modulo $q_2$ and vice versa,
\item $ \lfloor \frac{q_1q_2}{4}\rfloor- \lfloor \frac{q_1}{4}\rfloor- \lfloor \frac{q_2}{4}\rfloor \neq 0 \bmod 3,$ 
\item $\ell^3 \equiv 1 \bmod q_1$, but $\ell \neq 1 \bmod q_1$ and $\ell \neq 1 \bmod q_2$, 
\item $\ell$ has residue degree 3 and ramification degree $q_1q_2$ in $F/\Q$.
\end{itemize}
If Conjecture \ref{conj:main} holds, then every elliptic curve $E/\Q$ that has additive reduction at $\ell$ of Kodaira type III and good reduction at other primes that ramify in $F/\Q$, must have a non-trivial $p$-Selmer group for some prime $p\nmid [F\!:\!\Q]$.
\end{corollary}

\begin{proof}
First note that $\Gal(F/\Q)\simeq C_{q_1q_2}\!\rtimes\! C_3$ (the inertia group at $\ell$ is tame, so $C_{q_1q_2}$ is a subgroup, and the extension is split by the Schur-Zassenhaus theorem). As $\ell\neq 1\bmod q_1$, there is no Galois extension of $\Q$ of degree $q_1$ that is ramified at $\ell$, and similarly for $q_2$. It follows that $C_3$ must act non-trivially both on $C_{q_1}$ and $C_{q_2}$. Moreover, $q_1$ and $q_2$ are both cubes modulo each other and are $\equiv 4$ or $7 \bmod 9$, so Theorem \ref{thm:qqm} applies with $r=3$. 

If the Selmer group $\Sel_p(E/F)$ is trivial for a prime $p\nmid [F\!:\!\Q]$, then it is also trivial over any subfield of $F$
(see Lemma \ref{lem:basics}(\ref{lem:basics-selmer})), and hence neither the torsion nor $\sha$ contribute to the $p$-part of the $\BSD$ quotient in the theorem. The rank over $F$ is then also 0, so all the regulators are trivial. Thus by Theorem \ref{thm:qqm}, supposing that  $\Sel_p(E/F)=0$ for all $p\nmid [F\!:\!\Q]$,
$$
\frac{C_{E/F^{C_r}}\cdot C_{E/\Q}}{C_{E/F^{C_{q_1}\rtimes C_r}}\cdot C_{E/F^{C_{q_2}\rtimes C_r}}} = q_1^a q_2^b m
$$
for some $a,b\in\Z$ and $m\equiv 1 \mod q_1q_2$. 
However, in our setup this is not the case, as we now explain. 

First observe that all primes of good reduction, which include all ramified primes in $F/\Q$ apart from $\ell$, have trivial Tamagawa numbers and $|\omega/\omega^{\min}|$ terms in all extensions, and hence do not contribute to the ratio of the $C$-terms above.
If $q\neq\ell$ is a prime of bad reduction for $E/\Q$, then, by assumption, it is unramified in $F/\Q$. The minimal differential at $q$ then remains minimal in all extensions, so that the $|\omega/\omega^{\min}|$ terms for these primes are all 1. Moreover, the decomposition group at $q$ is either trivial or cyclic of order 3, $q_1, q_2$ or $q_1q_2$, and a straightforward case-by-case check shows that the Tamagawa numbers from the primes above $q$ contribute a perfect cube to the above ratio of $C$ terms (in fact each extension of $\Q_q$ always appears in the expression a multiple of 3 times).

Finally consider the primes above $\ell$. As $\ell$ has ramification degree $q_1q_2$ in $F/\Q$, it is totally ramified in $F^{C_r}, F^{C_{q_1}\rtimes C_r}$ and $F^{C_{q_2}\rtimes C_r}$. Thus $E$ has reduction type III or III$^*$ at the prime above $\ell$ in these fields, and the corresponding Tamagawa number is always 2 (see \cite{Silverman2}~\S~IV.9). The minimal model over $\Q_\ell$ does not remain minimal (valuation of the discriminant goes above 12, Lemma \ref{lem:basics}(\ref{lem:basics-pgomega})) and the $|\omega/\omega^{\min}|$ terms contribute
$
 \ell^{\lfloor \frac{3\cdot q_1q_2}{12} \rfloor}/\ell^{\lfloor \frac{3\cdot q_1}{12} \rfloor}\ell^{\lfloor \frac{3\cdot q_2}{12} \rfloor} = \ell^n
 $,
where $n\neq 0\bmod 3$ by assumption.

Putting these computations together shows that 
$$
\frac{C_{E/F^{C_r}}\cdot C_{E/\Q}}{C_{E/F^{C_{q_1}\rtimes C_r}}\cdot C_{E/F^{C_{q_2}\rtimes C_r}}} = x^3 \ell^n
$$
for some $x\in\Q$ and integer $n\neq 0 \bmod 3$. However, $\ell$ has order 3 in $\F_{q_1}^\times$ and so, as $q_1\neq 1\bmod 9$, it is not a cube in $\F_{q_1}$. As $q_2$ is a cube mod $q_1$, this expression cannot be of the form $q_1^a q_2^b m$ for any $m\equiv 1 \bmod q_1q_2$, which gives the desired contradiction.
\end{proof}

\begin{remark}\label{rmk:makeselmerexist}
Number fields satisfying the hypotheses of Corollary \ref{cor:makeselmer} do exist. For example, we can take $q_1=643$ and $q_2=43$. For simplicity, let us take $K=\Q(\zeta_9)^+$ (class number 1) and then choose $\ell$ and $F$ using class field theory as follows. First pick five candidate primes $\ell_1, \ldots \ell_5$ that satisfy the 3rd and 4th bullet points of the corollary --- these are congruence conditions, so such primes exist. As $\ell_i$ has residue degree 3 in $K$ and $\ell_i^3\equiv 1 \bmod q_1$, the group $(\cO_K/\prod \ell_i)^\times$ has a  $C_{q_1}^{\times5}$ quotient. The unit group of $K$ has rank 2, so the ray class group $((\cO_K/\prod \ell_i)^\times/$image of $\cO_K^\times)$ has a $\Gal(K/\Q)$-stable $C_{q_1}^{\times n}$ quotient for some $n\ge 3$. In particular, as $q_1\equiv 1 \bmod 3$ and $\F_{q_1}$ contains the 3rd roots of unity, there are at least three $\Gal(K/\Q)$-stable $C_{q_1}$-quotients whose corresponding fields $F_1, F_2, F_3$ under global class field theory are linearly disjoint. By construction, the $F_i$ are Galois over $\Q$, have degree $q_1$ over $K$ and only the $\ell_i$ can ramify in $F_i/K$. As $K$ has class number 1 and the $F_i$ are linearly disjoint, at least three of the $\ell_i$ must ramify in some of the fields. Now repeating the same construction with the same $\ell_i$ for $q_2$ similarly yields three fields $F_1', F_2', F_3'$ of degree $q_2$ over $K$. One of the $\ell_i$ must ramify both in one of the $F_i$ and in one of the $F_i'$, say $\ell_1$ ramifies in $F_1$ and in $F_2$. We can the take $F=F_1 F_2$ and $\ell=\ell_1$.
\end{remark}


\subsection{Forcing points of infinite order}
\label{ss:infiniteorder}

For our final type of application of Theorem \ref{thm:norm}, we will make the local data force the existence of points of infinite order on elliptic curves. This time, the idea is to make sure that the ratio of $\BSD$-terms in the theorem cannot be the norm of an element at all, and hence $E/F$ must have positive rank. In order to do this, we need a way of controlling the Tate--Shafarevich group. In general, this is very difficult, so we will simply make use of the fact that it has square order and that all squares are norms from quadratic fields.

\begin{theorem}
\label{thmquadrank}
Suppose Conjecture \ref{conj:main} holds. 
Let $E/\Q$ be an elliptic curve, $F/\Q$ a Galois extension with Galois group $G$, $\rho$ an irreducible representation of $G$ and
$$
  \qquad \Big(\bigoplus_{\fg\in\Gal(\Q(\rho)/\Q)} \!\!\!\rho^{\fg}\Big)^{\oplus m} 
  = \bigl(\bigoplus_i \Ind_{F_i/\Q} \triv\bigr) \ominus \bigl(\bigoplus_j \Ind_{F_j'/\Q} \triv\bigr)
$$
for some $m\in\Z$ and subfields $F_i, F_j'\subseteq F$.
If either
$
\frac{\prod\nolimits_i C_{E/F_i}}{\prod\nolimits_j C_{E/F_j'}}
$
is not a norm from some quadratic subfield $\Q(\sqrt{D})\subset \Q(\rho)$,
or if it is not a rational square when $m$ is even, 
then $E$ has a point of infinite order over $F$.
\end{theorem}

\begin{proof}
Suppose $\rk E/F=0$. By Theorem \ref{thm:norm},
$
\frac{\prod \BSD(E/F_i)}{\prod\BSD(E/F_j')}
$
is the $m$-th power of the norm of an element of $\Q(\rho)$. In particular it is a norm
from $\Q(\sqrt{D})$, and if $m$ is even it is a rational square.

As the rank is zero over $F$, the regulators that enter the $\BSD$-terms are all 1. The contributions from $\sha$ and torsion are all squares, and hence automatically norms from~$\Q(\sqrt{D})$. It follows that the remaining expression
$\frac{\prod\nolimits_i C_{E/F_i}}{\prod\nolimits_j C_{E/F_j'}}$
must be a norm from $\Q(\sqrt{D})$ as well, and a rational square in case $m$ is even.
\end{proof}

The criterion of Theorem \ref{thmquadrank} can be applied in many Galois groups to find local conditions on elliptic curves that guarantee the existence of points of infinite order. We illustrate it on the group of quaternions, $Q_8$:

\begin{corollary}
Suppose Conjecture \ref{conj:main} holds. 
Let $F/\Q$ be a Galois extension with Galois group $Q_8$. Then every elliptic curve $E/\Q$ with good reduction at 2 and 3 and with an odd number of potentially multiplicative primes that do not split in $F/\Q$ must have a point of infinite order over $F$.
\end{corollary}

\begin{proof}
Let $\rho$ be the 2-dimensional irreducible representation of $Q_8$, so that 
$$
  \rho^{\oplus 2}=\Ind_{1}^{Q_8}\triv \ominus \Ind_{C_2}^{Q_8}\triv.
$$
We will show that $\frac{C_{E/F}}{C_{E/L}}$ has odd 2-adic valuation, where $L=F^{C_2}$. The result then follows from the theorem.

Observe that if a prime $p$ splits in $F/\Q$, then it necessarily already splits in $L/\Q$. Indeed, if there is only one prime above $p$ in $L$, then the decomposition group at $p$ surjects onto $Q_8/C_2$. The only subgroup with this property is the whole of $Q_8$, so there is only one prime above $p$ in $F$. It follows that split primes contribute square contributions to $\frac{C_{E/F}}{C_{E/L}}$.

As $E$ has good reduction at $2$, these primes do not contribute to the ratio: $\omega$ remains minimal in all field extensions of $\Q_2$ and the local Tamagawa number is always 1 (Lemma~\ref{lem:basics}). At primes $v\nmid2$, the contribution from $\ord_2 |\omega/\omega^{\min}|_v$ will clearly be zero. Thus
$$
\ord_2 \frac{C_{E/F}}{C_{E/L}} = \sum_{p\in B} \ord_2(c_w/c_v),
$$
where $B$ is the set of primes of bad reduction of $E$ that do not split in $F/\Q$, and where $v$ and $w$ are the primes above $p$ in $L$ and $F$, respectively. 

If $p\in B$ then $p$ necessarily has residue degree 2 and ramification degree 4 in $F/\Q$ and the prime above it ramifies in $F/L$, as the only possible choice for the (tame!)\ inertia subgroup and its cyclic quotient is $Q_8/C_4=C_2$. In particular, if $p$ is a prime of potentially multiplicative reduction then $E$ has split multiplicative reduction at $v$ in $L$ and $c_w=2 c_v$ (Lemma \ref{lem:basics}).
If $p$ is a prime of potentially good reduction then the $p$-adic valuation $\delta$ of the minimal discriminant of $E/\Q_p$ determines the Kodaira type of $E$ at $v$ and at $w$. Recall that the Tamagawa number of $E$ over a local field $M$ is the number of Frobenius invariant points of $E(M^{nr})/E_0(M^{nr})$, so we read off from \cite{Silverman2} Chapter 9 Table 4.1  that the pair of Tamagawa numbers $c_v, c_w$ is either $3,3$ or $3,1$ ($\delta=2,10$), $1,1$ or $4,1$ ($\delta=3,9$, noting that $L_v$ is a quadratic unramified extension of a field, so Frobenius has odd order on $E/E_0$ over $L_v$), 2,2 $(\delta=4,8)$, or 1,1 $(\delta=6)$.
Thus in all cases of potentially good reduction $\ord_2 c_w/c_v$ is even. The result follows.
\end{proof}

As a final application, we will prove a result on the Birch--Swinnerton-Dyer conjecture in dihedral extensions. This time we will classify the cases when our local $\BSD$-term data predicts that $\rho$ appears in $E(F)_\C$ and compare it to the corresponding root number predictions.

\begin{theorem}\label{thm:parityconj}
Suppose Conjecture \ref{conj:main} holds. 
Let $F/\Q$ be a Galois extension with Galois group $D_{2pq}$, with $p,q\equiv 3\! \mod 4$ primes, and let $\rho$ be a faithful irreducible Artin representation that factors through $F/\Q$. Then for every semistable elliptic curve $E/\Q$,
$$
 \text{if } \ord_{s=1} L(E,\rho,s) \text{ is odd, then } \langle\rho,E(F)_\C\rangle>0.
$$
\end{theorem}

\begin{proof}
We first remark that this $L$-function does have an analytic continuation to $\C$ and satisfies the standard functional equation. 
(It can be expressed as a classical Rankin--Selberg product. Alternatively,  $\rho$ is induced from a 1-dimensional representation $\psi$ of $\Gal(F/K)$, where $K$ is the quadratic subfield of $F$, and so $L(E,\rho,s)=L(E/K,\psi,s)=L(\pi_{E/K}\otimes\psi,s)$, where $\pi_{E/K}$ is the automorphic form obtained by cyclic base change from the modular form attached to $E/\Q$ by modularity.)

We apply Theorem \ref{thm:norm} to the identity
$$
 \bigoplus_{\fg\in\Gal(\Q(\rho)/\Q)}\!\! \rho^{\fg}\>\>\>\>= \>\>\>\>\Ind^G_{C_2}\triv \>\ominus\> \Ind^G_{D_{2p}}\triv \> \ominus \>\Ind^G_{D_{2q}}\triv\> \oplus\> \triv,
$$
where $G=D_{2pq}$. Here $\Q(\rho)=\Q(\zeta_{pq})^+$ contains the quadratic field $\Q(\sqrt{pq})$. Since squares are always norms from quadratic fields we deduce that either $\langle\rho, E(F)_\C \rangle>0$ or
$$
 \frac{C_{E/L_{pq}}C_{E/\Q}}{C_{E/L_p} C_{E/L_q}}\cdot \frac{\Reg_{E/L_{pq}}\Reg_{E/\Q}}{\Reg_{E/L_{p}}\Reg_{E/L_{q}}} = N_{\Q(\sqrt{pq})/\Q}(x) \quad \text{for some } x\in\Q(\sqrt{pq}),
$$
where $L_{pq}=F^{C_2}$, $L_p=F^{D_{2q}}$ and  $L_q=F^{D_{2p}}$ are the intermediate fields of degree $pq, p$ and $q$ over $\Q$, respectively.

Write $M$ for the set of primes of multiplicative reduction of $E/\Q$, $s_r=1$ if the reduction at a prime $r\in M$ is split and $s_r=-1$ if it is non-split, and write $e_r$ and $f_r$ for the ramification and residue degree of a prime $r$ in $F/\Q$, respectively. Set
$$
X=\{v|\infty \text{ in } K\} \cup \{r\in M, e_r=1, f_r=2\}\cup \{r\in M, e_r=2, s_r=-1\}.
$$

\noindent {\bf Claim 1:}
If $\langle\rho, E(F)_\C \rangle=0$ then
$\frac{C_{E/L_{pq}}C_{E/\Q}}{C_{E/L_p}C_{E/L_q}}\cdot \frac{\Reg_{E/L_{pq}}\Reg_{E/\Q}}{\Reg_{E/L_{p}}\Reg_{E/L_{q}}} = (pq)^{\#X}\cdot\square$, where ``$\square$'' is shorthand for a rational square.

\noindent {\bf Claim 2:}
$\ord_{s=1}L(E,\rho,s) \equiv \#X \bmod 2$.

Observe that $pq$ is not the norm of any element of $\Q(\sqrt{pq})$. 
Indeed, as $p\equiv 3\bmod 4$, the norm equation $pq=a^2-pqb^2$ is not even soluble in $\Q_p$. 
It thus follows from the two claims and the formula above that $\langle\rho, E(F)_\C \rangle>0$, which proves the theorem.

{\em Proof of Claim 2:}
The parity of the order of vanishing of the $L$-function is given by the root number $w_{E,\rho}$.
As $E/\Q$ is semistable and $\dim\rho=2$, by \cite{Dokchitser2005} Thm.\ 1,
$$
 w_{E,\rho} = (-1)^{\dim\rho^-} \prod_{r\in M} s_r^{\dim\rho^{I_r}}\det(\Frob_r|\rho^{I_r}),
$$
where 
$\dim\rho^{-}$ is 1 or 2 according to whether $K$ is complex or real, $\Frob_r$ is any choice of Frobenius element at $r$ in $\Gal(F/\Q)$, and $\rho^{I_r}$ is the subspace of $\rho$ that is pointwise fixed by the inertia subgroup $I_r$.
If a prime $r\in M$ is unramified in $F/\Q$ then its contribution to the product is $-1$ if and only if $\Frob_r$ has order 2.
If it ramifies, then the contribution is $-1$ if and only if $I_r$ has order 2 (in $D_{2pq}$, $|I_r|=2$ forces $\Frob_r$ to be trivial) and $s_r=-1$. In other words $w_{E,\rho}=(-1)^{\#X}$.

{\em Proof of Claim 1:}
Since $E/\Q$ is semistable, its global minimal differential remains minimal in all field extensions, so we can write $C_{E/L}=\prod_{r\in M} C_{v|r}(L)$ with $C_{v|r}(L)=\prod_{v|r}c_v(E/L)$.

The group $D_{2pq}$ has five rational irreducible representations: trivial $\triv$, sign $\epsilon$, $\tau_p$ that factors through the $D_{2p}$-quotient and similarly $\tau_{q}$ and $\tau_{pq}$. 
Now pick points $P_1,\ldots, P_a\in E(\Q)$ that form a basis for $E(\Q)\otimes_\Z\Q$, the $\triv$-isotypical component of $E(F)\otimes\Q$. Complete it to a basis $P_1,\ldots, P_a, Q_1,\ldots Q_b\in E(L_p)$ for $E(L_p)\otimes\Q$, with the $Q_i$ belonging to the $\tau_p$-isotypical component of $E(F)\otimes\Q$; and similarly to $P_1,\ldots, P_a, R_1,\ldots R_b$ for $E(L_q)\otimes\Q$ with the $R_i$ belonging to the $\tau_q$-isotypical component.
By assumption, $\tau_{pq}$ does not appear in $E(F)\otimes\Q$, so that the $P_i, Q_i$ and $R_i$ together form a basis for $E(L_{pq})\otimes\Q$. Moreover, as the height pairing on $E(F)$ is Galois invariant, the spaces spanned by the $P_i$, the $Q_i$ and the $R_i$ are orthogonal to each other. Finally, recall that the height pairing scales under field extensions by the degree, so that the ratio of the regulators is
$$
  \frac{\Reg_{E/L_{pq}}\Reg_{E/\Q}}{\Reg_{E/L_{p}}\Reg_{E/L_{q}}} = \square \cdot q^{\rk E/L_p - \rk_{E/\Q}}p^{\rk E/L_q - \rk_{E/\Q}},
$$
the square error coming from the fact that our bases span finite index sublattices of $E(\Q)$, $E(L_p)$, $E(L_q)$ and $E(L_{pq})$ (see Lemma \ref{lem:basics}).

As the $p$- and $q$-primary parts of $\sha_{E/F}$ are finite (which is therefore also true over all the subfields), the
known cases of the parity conjecture for $E/\Q$, $E/L_p$ and $E/L_q$ (\cite{Dokchitsers2009} Thm.~1.3),
tell us that the parity of each exponent in the above formula is determined by the corresponding root number. Thus, like for the $C$ terms, we can express this as a product
$$
  q^{\rk E/L_p - \rk E/\Q} = \square \cdot \prod_{r\in M\cup\{\infty\}} q^{\theta_r^{(q)}},
$$
where $\theta_r^{(q)}$ is 0 or 1 depending on whether $w(E/\Q_r)\prod_{v|r}w(E/(L_{p})_v)$ is 1 or $-1$; and similarly for the exponent of $p$. 
Hence
$$
\frac{C_{E/L_{pq}}C_{E/\Q}}{C_{E/L_p}C_{E/L_q}}\cdot \frac{\Reg_{E/L_{pq}}\Reg_{E/\Q}}{\Reg_{E/L_{p}}\Reg_{E/L_{q}}} = \square\cdot  \prod_{r\in M\cup\{\infty\}} Z_r,
$$
where $Z_r=C_{v|r}(\Q)C_{v|r}(L_p)C_{v|r}(L_q)C_{v|r}(L_{pq}) q^{\theta_r^{q}}p^{\theta_r^{p}}$. Thus it now suffices to check that $Z_r=pq\square$ for $r\in X$ and $Z_r=\square$ for $r\notin X$. 

To explicitly determine $Z_r$, we systematically work through all possibilities. Recall that the local root number $w(E/L_v)$ is $+1$ for good and non-split multiplicative reduction and $-1$ for split multiplicative reduction and for archimedean places. Recall also that if the Kodaira type of $E/L_v$ is I$_n$ then the Tamagawa number is $n$ if the reduction is split, and 1 or 2 if it is non-split, depending on whether $n$ is odd or even, which we will denote by $\tilde{n}$. Finally, multiplicative reduction of type I$_n$ becomes of type I$_{en}$ after a ramified extension of degree $e$, split reduction remains split, and non-split reduction becomes split if the extension has even residue degree. We now tabulate the contribution to the above product from a prime $r$ depending on its ramification degree $e_r$ and residue degree $f_r$ in $F/\Q$; in $D_{2pq}$ these uniquely determine the inertia and decomposition subgroups, and hence the splitting behaviour of $r$ in all intermediate extension. The values are constrained by the fact that both the (tame!) inertia group is cyclic of order $e_r$ and normal in the decomposition group with a cyclic quotient of order $f_r$. The entries for split and non-split multiplicative reduction of type I$_n$ are separated by a ``;''.
$$
\hskip-1.2cm
\begin{array}{|rl|c|ll|ll|ll|c|c|c|}
\hline
e_r, &f_r & C_{v|r}(\Q)&C_{v|r}(L_p) &&C_{v|r}(L_q)&& C_{v|r}(L_{pq})&& q^{\theta^{(q)}_r} & p^{\theta^{(p)}_r} & Z_r \cr
\hline
1,&1 &  n;\tilde{n} &n^p;&\tilde{n}^p &  n^q;&\tilde{n}^q& n^{pq};&\tilde{n}^{pq} & 1;1 & 1;1 & \square \cr
1,&2 &  n;\tilde{n}&n^{\frac{p+1}2} ;&\tilde{n}n^{\frac{p-1}2}& n^{\frac{q+1}2} ;&\tilde{n}n^{\frac{q-1}2}& n^{\frac{pq+1}2} ;&\tilde{n}n^{\frac{pq-1}2} &q;q & p;p & pq\square  \cr
1,&p&  n;\tilde{n} &  n;&\tilde{n} &  n^q;&\tilde{n}^q&  n^q;&\tilde{n}^q & 1;1 & 1;1& \square\cr
1,&q&  n;\tilde{n} &  n;&\tilde{n} &  n^p;&\tilde{n}^p&  n^p;&\tilde{n}^p & 1;1 & 1;1& \square\cr
1,&pq&  n;\tilde{n}& n;&\tilde{n}& n;&\tilde{n}& n;&\tilde{n} & 1;1 & 1;1 & \square \cr
2,&1&  n;\tilde{n}& (2n)^{\frac{p-1}2}n;&2^{\frac{p-1}2}\tilde{n}&(2n)^{\frac{q-1}2}n;&2^{\frac{q-1}2}\tilde{n} &(2n)^{\frac{pq-1}2}n;&2^{\frac{pq-1}2}\tilde{n} & q;1 & p;1 & pq\square;\square\cr
p,&1& n;\tilde{n} &np;&\tilde{n}&n^q;&\tilde{n}^q&(np)^q;&\tilde{n}^q & 1;1 & 1;1 & \square \cr
p,&2&  n;\tilde{n}&np;&\tilde{n}&  n^{\frac{q+1}2} ;&\tilde{n}n^{\frac{q-1}2} &  (np)^{\frac{q+1}2} ;&\tilde{n}(np)^{\frac{q-1}2} & 1;1 & p;p & \square\cr
p,&q&  n;\tilde{n} & np;&\tilde{n} & n;&\tilde{n} & np;&\tilde{n} & 1;1 & 1;1 & \square\cr
q,&1& n;\tilde{n} &nq;&\tilde{n}&n^p;&\tilde{n}^p&(nq)^p;&\tilde{n}^p & 1;1 & 1;1 & \square \cr
q,&2&  n;\tilde{n}&nq;&\tilde{n}&  n^{\frac{p+1}2} ;&\tilde{n}n^{\frac{p-1}2} &  (nq)^{\frac{p+1}2} ;&\tilde{n}(nq)^{\frac{p-1}2} & 1;1 & q;q & \square\cr
q,&p&  n;\tilde{n} & nq;&\tilde{n} & n;&\tilde{n} & nq;&\tilde{n} & 1;1 & 1;1 & \square\cr
2p,&1&  n;\tilde{n} & np;&\tilde{n} & (2n)^{\frac{q-1}2}n;&2^{\frac{q-1}2}\tilde{n} &  (2np)^{\frac{q-1}2}np;&2^{\frac{q-1}2}\tilde{n} & 1;1 & p;1 & \square\cr
2q,&1&  n;\tilde{n} & np;&\tilde{n} & (2n)^{\frac{p-1}2}n;&2^{\frac{p-1}2}\tilde{n} &  (2nq)^{\frac{p-1}2}nq;&2^{\frac{p-1}2}\tilde{n} & 1;1 & q;1 & \square\cr
pq,&1&  n;\tilde{n} & np;&\tilde{n} & nq;&\tilde{n} & npq;&\tilde{n}& 1;1 & 1;1 & \square\cr
pq,&2&  n;\tilde{n}& np;&\tilde{n} & nq;&\tilde{n} & npq;&\tilde{n}  & 1;1 & 1;1 & \square\cr
\hline
\end{array}
$$
Finally, note that if the quadratic field $K$ is real, then $F/\Q$ is totally real, so $L_p$ has $p$ infinite places and $q^{\theta_\infty^{q}}=1$, and similarly for $L_q$; hence $Z_\infty=1$. If $K$ is imaginary, then $L_p$ has one real and $\frac{p-1}2$ (=odd) complex places and $q^{\theta_\infty^{q}}=q$, and similarly for $L_q$; hence $Z_\infty=pq$.
Thus, indeed, $Z_r=pq\square$ for $r\in X$ and $Z_r=\square$ for $r\notin X$, as required.
\end{proof}

\subsection{Summary of some basic properties}\label{ss:basic}

We list some standard results regarding elliptic curves over local and global fields. We give brief proofs as, while these results are well-known, they may not always be easy to find in the literature.

\begin{lemma}\label{lem:basics}
Let $E/K$ be an elliptic curve over a number field, $F/K$ a field extension of finite degree $d$. Let $v$ be a finite place of $K$ with $w|v$ a place above it in $F$, and $\omega_v$ and $\omega_w$ minimal differentials for $E/K_v$ and $E/F_w$, respectively.
\begin{enumerate}
\item\label{lem:basics-selmer} If $F/K$ is Galois, then $\Sel_{n}(E/K)$ is a subgroup of $\Sel_{n}(E/F)$ for all $n$ coprime~to~$d$.
\item\label{lem:basics-height} For $P, Q\in E(K)$, their N\'eron--Tate height pairings over $K$ and $F$ are related by $\langle P,Q\rangle_F=d\langle P,Q\rangle_K$.
\item\label{lem:basics-regulator} If $\rk E/F=\rk E/K$, then $\Reg_{E/F}\!=\!\frac{d^{\rk E/K}}{n^2} \Reg_{E/K}$, where $n$ is the index of $E(K)$ in $E(F)$.
\item\label{lem:basics-semistabletamagawa} If $E/K_v$ has good reduction then $c_v=1$. If $E/K_v$ has multiplicative reduction of Kodaira type $\text{I}_n$ then $n=\ord_v\Delta^{\text{min}}_{E,v}$ and $c_v=n$ if the reduction is split, and $c_v=1$ (respectively, $2$) if the reduction is non-split and $n$ is odd (respectively, even).
\item\label{lem:basics-semistableomega} If $E/K_v$ has good or multiplicative reduction then $|\omega_v/\omega_w|_w=1$.
\item\label{lem:basics-pgomega} If $E/K_v$ has potentially good reduction and the residue characteristic is not 2 or 3, then $|{\omega_v}/{\omega_w}|_w=q^{\lfloor\frac{e_{F/K}\ord_v\Delta_{E,v}^{\rm{min}}}{12}\rfloor}$, where $q$ is the size of the residue field at $w$.
\item\label{lem:basics-pm}  If $v$ has odd residue characteristic, $E/K_v$ has potentially multiplicative reduction and $F_w/K_v$ has even ramification degree, then $E/F_w$ has multiplicative reduction.
\item \label{lem:basics-m} Multiplicative reduction becomes split after a quadratic unramified extension.
\end{enumerate}
\end{lemma}

\begin{proof}
(\ref{lem:basics-selmer})
In the inflation-restiction sequence $H^1(\Gal(F/K),E(F)[n])\to H^1(K,E[n])\to H^1(F,E[n])$, 
the first term is killed both by $|\!\Gal(F/K)|$ and by $n$, and is therefore trivial. Thus the second map and its restriction to $n$-Selmer groups are injective.

(\ref{lem:basics-height}) This follows from the definition of the height pairing, see \cite{TateBSD} (1.6). (Note that it is {\em not} normalised as for the absolute height.)

(\ref{lem:basics-regulator}) Follows from (\ref{lem:basics-height}) and the fact that the height pairing is bilinear and non-degenerate.

(\ref{lem:basics-semistabletamagawa}) \cite{Silverman2} \S IV.9.

(\ref{lem:basics-semistableomega}) As $E/K_v$ has good or multiplicative reduction, its minimal Weierstrass model over~$K_v$ remains minimal over $F_w$, so $\omega_v$ is also a minimal differential over $F$.

(\ref{lem:basics-pgomega}) 
In this setting $\ord_w\Delta_{E,w}^{\text{min}}<12$, so the result follows from the formula in Notation~\ref{def:omega}.

(\ref{lem:basics-pm}) This follows from the theory of the Tate curve, see e.g. \cite{Silverman2} Exc.\ 5.11.

(\ref{lem:basics-m}) Clear from the definition of non-split multiplicative reduction.
\end{proof}


\section{Arithmetically similar twists with different $L$-values}\label{s:hanneke}

In this section we discuss the problem of formulating a precise Birch--Swinnerton-Dyer type formula for twists of elliptic curves by Dirichlet characters $\chi$. We will make the information that we know about $\sL(E,\chi)$ explicit and discuss the difficulties illustrated in Example \ref{ex:introha}. We will also give many numerical examples, for the benefit of those readers who may wish to analyse these $L$-values in more detail. 

The numerical examples throughout this section were worked out using Magma \cite{Magma}. The orders of $\sha$ given are strictly speaking ``analytic orders of $\sha$'', that is the orders that are predicted by the Birch--Swinnerton-Dyer conjecture. 

\begin{notation}
Recall from Notation \ref{not:conventions} that we identify Dirichlet characters $\chi$ 
with their corresponding $1$-dimensional Galois representations. We write $K^\chi$ for the abelian number field cut out by the kernel of $\chi$, that is for the smallest extension $K^\chi/\Q$ such that $\chi$ factors through $\Gal(K^\chi/\Q)$. 
\end{notation}

In the context of Dirichlet characters, we already know from Theorem \ref{thm:introlfunctions} a substantial amount about $L(E,\chi,1)$ in terms of arithmetic data:

\begin{theorem}\label{thm:hapredict}
Suppose Stevens's Manin constant conjecture holds for $E/\Q$. Let $\chi$ be a non-trivial primitive Dirichlet character of order $d$ and conductor coprime to $\mathfrak{f}_E$. Then $\sL(E,\chi)\in \Z[\zeta_d]$ and, if $L(E,\chi,1)\neq 0$, then furthermore
$$
\zeta\cdot\sL(E,\chi)\in \R, \qquad\quad \text{for }\> \zeta=\chi(\f_E)^{\frac{d+1}{2}}\sqrt{ \chi(-1) w_E}.
$$ 
If $\rk E/\Q=0$ and the Birch--Swinnerton-Dyer conjecture holds for $E$ over $\Q$ and $K^\chi$, then
$$
 N_{\Q(\zeta_d)^+/\Q} (\zeta\cdot \sL(E,\chi)) = 
 \pm \frac
 {|E(\Q)_{\tors}|}
 {|E(K^\chi)_{\tors}|}
 \sqrt{\frac
 {|\sha_{E/K^\chi}|\prod_v c_v(E/K^\chi)}
 {|\sha_{E/\Q}|\prod_p c_p(E/\Q)}
 }.
$$
If moreover $d$ is odd and $\BSD(E/K^\chi)=\BSD(E/\Q)$ then $\sL(E,\chi)=\zeta^{-1}u$ for some unit $u\!\in\!\cO_{\Q(\zeta_d)^+}^\times$.
\end{theorem}

\begin{proof}
The first claim follows from Theorem \ref{thm:introlfunctions} (8,9,10) with $\rho=\chi$.

Applying Theorem~\ref{thm:introlfunctions}(12) with the identity $\bigoplus_{\fg\in\Gal(\Q(\zeta_d)/\Q)}\chi^\fg \oplus\triv = \C[G]$ shows that $ N_{\Q(\zeta_d)^+/\Q} (\zeta\cdot \sL(E,\chi)) = \pm \sqrt\frac{\BSD(E/K^\chi)}{\BSD(E/\Q)}$. Since the conductor of $E$ is coprime to that of $\chi$, the primes of bad reduction of $E$ are unramified in $K^\chi/\Q$, so a global minimal differential for $E/\Q$ remains minimal over $K^\chi$ and hence all the contributions of the form $|\omega/\omega^{\min}|$ to the $\BSD$-terms are trivial. This proves the desired second formula.

For the final claim, note that as $\rk E/\Q\! =\! 0$ and $d$ is odd, we must have $w_E\!=\!\chi(-1)\!=\!1$, so that $\zeta\in\Q(\zeta_d)$. The result now follows from the previous parts.

\end{proof}

Under the above assumptions, we can predict the value $L(E,\chi,1)$ from Birch--Swinnerton-Dyer type information up to an element of norm $\pm 1$ in $\Q(\zeta_d)^+$. In fact, since $\sL(E,\chi)$ is integral, the prediction is stronger than that. For instance, if $\chi$ has order 3 and $\BSD(E/K^\chi)=\BSD(E/\Q)$ then $L(E,\chi,1)$ is fully determined up to a sign. However, this final ambiguity appears to be severe:

\begin{theorem}\label{thm:noway}
For elliptic curves $E/\Q$ and Dirichlet characters $\chi$ as in Theorem \ref{thm:hapredict}, 
\begin{enumerate}[leftmargin=*]
\item 
$\sL(E,\chi)$ cannot be expressed purely as a function of $\chi$, of $E(\Q)$, $\sha_{E/\Q}$, $\prod_p E(\Q_p)/E_0(\Q_p)$ as abelian groups and of $E(K^\chi), \sha_{E/K^\chi}, \prod_v E(K^\chi_v)/E_0(K^\chi_v)$ as $\Gal(K^\chi/\Q)$-modules. 
\item The fractional ideal $(\sL(E,\chi))$  cannot be expressed purely as a function of $\chi$, and~of~$E(\Q)$, $\sha_{E/\Q}$, $\prod_p E(\Q_p)/E_0(\Q_p)$, $E(K^\chi), \sha_{E/K^\chi}$ and $\prod_v E(K^\chi_v)/E_0(K^\chi_v)$ as abelian groups. 
\end{enumerate}
Here the products are taken over all primes of $\Q$ and of $K^\chi$, and $E_0$ denotes the usual subgroup of points of non-singular reduction.
\end{theorem}

This theorem follows from the fact that one can find curves with identical arithmetic invariants listed in (1) and (2), but with different modified $L$-values $\sL(E,\chi)$. This is shown by the next two examples, where most of the objects listed in (1) and (2) are trivial. 

\begin{example}\label{ex:ha5}
Let $E_1/\Q$ be the elliptic curve given by
\[
y^2 + y = x^3 - 8x - 9,
\]
and $E_2/\Q$ be another elliptic curve given by
\[
y^2 + y = x^3 + x - 1,
\]
which have Cremona labels 307a1 and 307c1, respectively. 
Let $\chi$ be the primitive Dirichlet character of order 5 and conductor 11 defined by $\chi(2)=\zeta_5$. 
Both curves have 
$$
 |E_i(\Q)|=|E_i(K^\chi)|=|\sha_{E_i/\Q}|=|\sha_{E_i/K^\chi}|=\prod_p c_p(E_i/\Q)=\prod_v c_v(E_i/K^\chi)=1.
$$
In particular, all the groups listed in Theorem \ref{thm:noway} are trivial.
In fact, the curves also have the same conductor $\mathfrak{f}_{E_i}=307$ and the same discriminant $\Delta_{E_i}=-307$. 
However, their modified $L$-values differ:
\begin{align*}
\sL(E_1,\chi)=1, \quad \sL(E_2,\chi)=\zeta_5^4 (1+\zeta_5)^2.
\end{align*}
\end{example}

\begin{remark}
As the discriminants for the two curves in the above example are the same and thus in particular have the same sign, both curves have the same number of connected components over $\R$. In other words, one can add the group of real connected components $E(\R)/E_0(\R)$ to the list of groups in Theorem~\ref{thm:noway}(1), as well as the conductor and the discriminant of $E$.
\end{remark}

\begin{example}\label{ex:haprime}
Let $E_1/\Q$ be the elliptic curve given by
\[
y^2+y=x^3-x^2-1,
\]
and $E_2/\Q$ be another elliptic curve given by
\[
y^2+xy = x^3 +x^2-3x-4,
\]
which have Cremona labels 291d1 and 139a1, respectively. Let $\chi$ be the primitive character of order five and conductor 31 defined by ${\chi(3)=\zeta_5^3}$. 
Both curves have 
$$
 |E_i(\Q)|=|E_i(K^\chi)|=|\sha_{E_i/\Q}|=\prod_p c_p(E_i/\Q)=\prod_v c_v(E_i/K^\chi)=1 \quad\text{and}\quad |\sha_{E_i/K^\chi}|=11^2.
$$
The discriminants $\Delta_{E_1}=-291$ and $\Delta_{E_2}=-139$ again have the same sign. For these curves,
$$
  \sL(E_1,\chi)=2\zeta_5^3 -\zeta_5^2 - \zeta_5 + 2 \qquad \text{and}\qquad \sL(E_2,\chi)=5\zeta_5^3 + \zeta_5^2 + \zeta_5 + 5.
$$
These factorise as
$$
  (\sL(E_1,\chi))=\p_1 \p_2 \qquad\text{and}\qquad (\sL(E_2,\chi))=\p_3 \p_4,
$$
where
$\p_1=(11,7+\zeta_5)$, $\p_2=(11,8+\zeta_5)$, $\p_3=(11,6+\zeta_5)$, $\p_4=(11,2+\zeta_5)$ are the primes of $\Q(\zeta_5)$ above 11.

We note that it is plausible that the exact factorisation can be recovered from the Galois module structure of $\sha$. Unfortunately, it appears to be beyond our computational reach to check this at present. (See, however, the recent work of Burns and Castillo \cite{BurnsCastillo} Rmk. 7.4.)
\end{example}

\begin{remark}\label{rem:primesplit}
In the above example, our results on $L$-values are strong enough to predict that the ideal $\sL(E_i,\chi)$ must be either $\p_1 \p_2$ or $\p_3 \p_4$, though, as the example illustrates, they do not allow us decide which of the two occurs. To see why the factorisation must be one of these two, consider any Dirichlet character $\chi$ of order 5 and any elliptic curve $E/\Q$ satisfying the conditions of Theorem \ref{thm:hapredict} and additionally $\frac{\BSD(E/K^\chi)}{\BSD(E/\Q)}=11^2$.
Then by Theorem \ref{thm:introlfunctions} (10), (11) and (6), $\sL(E,\chi)$ is an element of $\Z[\zeta_5]$ of norm $11^2$ and generates an ideal that is fixed by complex conjugation. Hence $(\sL(E,\chi))$ must be either $\p_1 \p_2$ or $\p_3 \p_4$.
\end{remark}

For those who may be interested in investigating these $L$-values further, we end by giving a range of further examples. All elliptic curves below are given by their Cremona labels.

\begin{example}
There are plenty of curves that have trivial Mordell--Weil groups, $\sha$ and Tamagawa numbers both over $\Q$ and over $K^\chi$ for the same Dirichlet character $\chi$ of order 5  as in Example \ref{ex:ha5}. Here we have chosen some groups of such curves that also have the same conductors, but, as in the example, have different modified $L$-values  (here $u=1+\zeta_5$ is a fundamental unit in $\Q(\zeta_5)$). 

\begingroup\smaller[1]
\begin{table}[h!]
\begin{tabular}{||cc||cc||cc||cc||cc||}
\hline
$E$ & $\!\!\!\sL(E,\chi)\!\!\!$ & $E$ & $\!\!\!\sL(E,\chi)\!\!\!$ & $E$ & $\!\!\!\sL(E,\chi)\!\!\!$  & $E$ & $\!\!\!\sL(E,\chi)\!\!\!$ & $E$ & $\!\!\!\sL(E,\chi)\!\!\!$\\
\hline
307a1 & $1$ & 432g1 &$u^2$& 714b1 & $1$ & 1187a1&   $\zeta_5^3 u^{-1}$ & 1216g1 & $-\zeta_5^2 u^2$ \\
307c1 & $\zeta_5^4 u^2$ &  432h1 & $-\zeta_5^4 u^{-1}$ & 714h1 & $-\zeta_5u^3$ &1187b1&   $\zeta_5^4 u^{-3}$ & 1216k1 & $\zeta_5  u^{-1}$ \\
\hline
\end{tabular}
\caption{Conductor $\f_\chi=11$ with $\chi(2)=\zeta_5, \Delta_{K^{\chi}}=11^4$.}
\end{table}
\endgroup
\end{example}

\begin{example}
The examples are even easier to find for cubic characters $\chi$. As before, we will look at curves with
$$
 |E(\Q)|=|E(K^\chi)|=|\sha_{E/\Q}|=|\sha_{E/K^\chi}|=\prod_p c_p(E/\Q)=\prod_v c_v(E/K^\chi)=1.
$$
All of the curves we look at will satisfy the conditions of Theorem \ref{thm:hapredict}, and thus by the same theorem we can predict the $L$-values up to sign. 
How to predict the sign is unclear, even for curves with the same conductor.

\begingroup\smaller[2]
\begin{table}[H]
\begin{tabular}{||cc||cc||cc||cc||cc||}
\hline
$E$ & $\!\!\!\sL(E,\chi)\!\!\!$ & $E$ & $\!\!\!\sL(E,\chi)\!\!\!$ & $E$ & $\!\!\!\sL(E,\chi)\!\!\!$ & $E$ & $\!\!\!\sL(E,\chi)\!\!\!$ & $E$ & $\!\!\!\sL(E,\chi)\!\!\!$ \\
\hline
\hline
1356d1 &$\zeta_3$& 3264r1 & $-\zeta_3$ & 3540a1&   $-\zeta_3$ & 4800i1 & $-\zeta_3$ && \\
1356f1 &$-\zeta_3$&  3264s1& $\zeta_3$& 3540b1 & $\zeta_3$& 4800bj1& $-\zeta_3$ &&\\
& & &&&& 4800bm1& $ \zeta_3$ &&\\
\hline
 \multicolumn{10}{|c|}{Conductor $\f_\chi=7$ with $\chi(3)=\zeta_3^2$, $\Delta_{K^{\chi}}=49$.}\\
\hline
\hline
222b1 & $-1$ &  1392c1 & $-1$ & 4386c1&   $-1$ & 9024l1 & $-\zeta_3^2$  && \\
222e1 & $1$ &  1392j1& $1$& 4386m1 & 1 & 9024bf1& $\zeta_3^2$ & &\\
\hline
 \multicolumn{10}{|c|}{Conductor $\f_\chi=13$ with $\chi(2)=\zeta_3^2$, $\Delta_{K^{\chi}}=169$.}\\
\hline
\hline
702d1 & $-1$ & 1443a1 &$1$& 5616j1 & $-1$ & 12096bq1&   $1$ & 19008u1 & $-1$  \\
702i1 & $1$ & 1443b1 &$-1$& 5616o1 & $1$ & 12096dc1&   $-1$ & 19008bh1 & $1$   \\
 & &  && 5616p1 & $1$ & 12096dd1&   $1$ &&   \\
\hline
  \multicolumn{10}{|c|}{Conductor $\f_\chi=19$ with $\chi(2)=\zeta_3^2$, $\Delta_{K^{\chi}}=381$.}\\
\hline
\hline
714b1 & $-1$ &  2453a1 & $1$ & 8138b1 & 1 &  12096x1 & $\zeta_3$  &&\\
714h1 & $1$ &  2453c1 & $-1$ & 8138c1 & $-1$ &  12096dc1 & $\zeta_3$ &&\\
& & &&&& 12096dd1& $ -\zeta_3$ && \\
\hline
  \multicolumn{10}{|c|}{Conductor $\f_\chi=31$ with $\chi(3)=\zeta_3$, $\Delta_{K^{\chi}}=961$.}\\
\hline
\hline
5885a1 & $-\zeta_3$ & 11764a1 &$-\zeta_3$& 12096x1 & $\zeta_3^2$ & 15498h1 & $-\zeta_3^2$ & 16590c1&   $1$  \\
5885d1 & $\zeta_3$ & 11764b1 &$\zeta_3$& 12096bb1 & $-\zeta_3^2$ & 15498i1 & $\zeta_3^2$ & 16590n1&   $-1$  \\
 &  &  && 12096bn1 & $\zeta_3^2$ &  &  & &    \\
 &  &  && 12096cz1 & $-\zeta_3^2$ &  &  & &     \\
\hline
    \multicolumn{10}{|c|}{Conductor $\f_\chi=37$ with $\chi(2)=\zeta_3$, $\Delta_{K^{\chi}}=1369$}\\
\hline
\end{tabular}
\phantom{$X^x$}
\caption{Modified $L$-values for varying Dirichlet characters $\chi$ of order three.}
\end{table}
\endgroup

\par In these examples, the curves in each block also have discriminants of the same sign as each other and the same number of points over $\F_3$. The first condition ensures that they have the same number of real components. The second condition is motivated by $p$-adic $L$-functions, where the interpolation formula for $L$-values is adjusted by an extra term that depends on $|E(\F_p)|$.
\end{example}

\begin{example}
Here we give a list of curves similar to Example \ref{ex:haprime}. We again take the character $\chi$ of order five and conductor 31 defined by ${\chi(3)=\zeta_5^3}$, and consider curves with conductor coprime to 31 with
$$
 |E(\Q)|=|E(K^\chi)|=|\sha_{E/\Q}|=\prod_p c_p(E/\Q)=\prod_v c_v(E/K^\chi)=1 \quad\text{and}\quad |\sha_{E/K^\chi}|=11^2.
$$
We know from Remark \ref{rem:primesplit} that the ideal $(\sL(E,\chi))$ of $\cO_{\Q(\zeta_5)}$ is either $\p_1\p_2$ or $\p_3\p_4$, where $\p_1=(11,7+\zeta_5), \p_2=(11,8+\zeta_5),\p_3=(11,6+\zeta_5), \p_4=(11,2+\zeta_5)$ are the primes of $\Q(\zeta_5)$ above 11. For the following list of curves list $\sL(E,\chi)$ splits as $\p_1 \p_2$: 216b1, 216c1, 291d1, 443c1, 475a1. For the following list of curves $\sL(E,\chi)$ splits as $\p_3 \p_4$: 139a1, 140b1, 267b1, 333d1, 378h1, 432g1, 579a1. 
\end{example}



\begin{thebibliography}{99}

\bibitem{Bartel}
A. Bartel, Large Selmer groups over number fields, Math. Proc. Cambridge Philos. Soc. 148 (2010), no.~1, 73--86. 

\bibitem{Magma}
W. Bosma, J. Cannon, C. Playoust,
The Magma algebra system. I: The user language, J. Symb. Comput. 24, No. 3--4 (1997), 235--265.

\bibitem{BurnsCastillo}
D. Burns, D. M. Castillo,
On refined conjectures of Birch and Swinnerton-Dyer type for Hasse--Weil--Artin $L$-series,
arxiv: 1909.03959.


\bibitem{Coates1991}
J. Coates, Motivic $p$-adic $L$-functions, in: $L$-functions and arithmetic (Durham, 1989), 141--172, London Math. Soc. Lecture Note Ser., 153, Cambridge Univ. Press, Cambridge, 1991.


\bibitem{Dokchitser2005} 
V. Dokchitser, Root numbers of non-abelian twists of elliptic curves, Proc. London Math. Soc. (3) 91 (2005), 300--324.

\bibitem{Dokchitsers2009} 
T. Dokchitser, V. Dokchitser, Regulator constants and the parity conjecture, Invent. Math. 178 (2009), 23--71.

\bibitem{Deligne1979}
P. Deligne, Valeurs de fonctions L et p\'{e}riodes d'int\'{e}grales, in: Automorphic forms, representations and L-functions, Part 2 (ed. A. Borel and W. Casselman), Proc. Symp. in Pure Math. 33 (AMS, Providence, RI, 1979) 313--346.

\bibitem{Martinet1977}
J. Martinet, Character Theory and Artin L-functions, in: Algebraic Number Field (ed. A. Frohlich), Proc. of the Durham Symp. (Academic Press, London, 1975).

\bibitem{Matsuno}
K. Matsuno, Elliptic curves with large Tate--Shafarevich groups over a number field,
Math. Research Letters 16 (2009), no. 3, 449--461.

\bibitem{Rohrlich1990}
D. Rohrlich, The vanishing of certain Rankin-Selberg convolutions, in: Automorphic Forms and Analytic Number Theory, 123--133. Les publications CRM, Montreal (1990).

\bibitem{Silverman2}
J. H. Silverman, Advanced topics in the arithmetic of elliptic curves, GTM 151, Springer, Berlin (1994).

\bibitem{Stevens}
G. Stevens,  Stickelberger elements and modular parametrizations of elliptic curves,  Invent. Math. 98 (1989), no. 1, 75--106.

\bibitem{TateBSD}
J. Tate , On the conjectures of Birch and Swinnerton--Dyer and a geometric analog, S\'eminaire Bourbaki, 18e ann\'ee, 1965/66, no. 306.

\bibitem{TatN}
J. Tate, Number theoretic background, in: Automorphic forms, representations and L-functions, Part 2 (ed. A. Borel and W. Casselman), Proc. Symp. in Pure Math. 33 (AMS, Providence, RI, 1979) 3--26.

\bibitem{Venjakob2007}
O. Venjakob, From the Birch and Swinnerton-Dyer conjecture over the equivariant Tamagawa number conjecture to non-commutative Iwasawa theory - a survey, in: L-functions and Galois representations, London Math. Soc. Lecture Note (Cambridge University Press, 2007) 333--380.

\bibitem{WW}
H. Wiersema, C. Wuthrich, Integrality of twisted L-values of elliptic curves, arxiv:2004.05492.

\end{thebibliography}
\end{document}